\documentclass[a4paper,10pt]{scrartcl}
\usepackage[utf8]{inputenc}
\usepackage[english]{babel}
\usepackage[fixlanguage]{babelbib}
\usepackage{cite}
\usepackage{amssymb, amsmath, amstext, amsopn, amsthm, amscd, amsxtra, amsfonts}

\usepackage{paralist}
\usepackage[T1]{fontenc}
\usepackage{ifthen}
\usepackage{colonequals}

\usepackage[all]{xy}
\SelectTips{cm}{}
\usepackage{tikz}

\usepackage[bookmarks,bookmarksnumbered,plainpages=false]{hyperref}

\swapnumbers
\newtheoremstyle{standard}
 {16pt}  
 {16pt}  
 {}  
 {}  
 {\bfseries}
 {}  
 { } 
 {{\thmname{#1~}}{\thmnumber{#2.}}\thmnote{~(#3)}} 

\newtheoremstyle{kursiv}
 {16pt}  
 {16pt}  
 {\itshape}  
 {}  
 {\bfseries}
 {}  
 { } 
 {{\thmname{#1~}}{\thmnumber{#2.}}\thmnote{~(#3)}} 

\theoremstyle{standard}
\newtheorem{defn} [subsection]{Definition}
\newtheorem{ex} [subsection]{Example}
\newtheorem{rem}   [subsection]{Remark}
\newtheorem{nota}   [subsection]{Notation}
\newtheorem{setup} [subsection]{}

\theoremstyle{definition}

\theoremstyle{kursiv}
\newtheorem{thm}[subsection]{Theorem}
\newtheorem{prop} [subsection]{Proposition}
\newtheorem{cor} [subsection]{Corollary}
\newtheorem{lem} [subsection]{Lemma}

\setdefaultenum{\normalfont (a)}{i.}{A.}{1.}

\newcommand{\Evol}{\mathrm{Evol}}
\newcommand{\im}{\mathrm{im}}

\newcommand{\evol}{\mathrm{evol}}
\newcommand{\ev}{\mathrm{ev}}

\newcommand{\di}{\mathrm{d}}
\newcommand{\id}{\mathrm{id}}

\newcommand{\N}{\mathbb{N}}

\newcommand{\R}{\mathbb{R}}

\newcommand{\K}{\mathbb{K}}
\newcommand{\C}{\mathbb{C}}

\newcommand{\LB}[1][\cdot \hspace{1pt} , \cdot]{[\hspace{1pt} #1 \hspace{1pt} ]}

\newcommand{\BHol}{\Hol_{\mathrm{b}}}
\newcommand{\Hol}{\mathrm{Hol}}

\newcommand{\norm}[1]{\left\lVert #1 \right\rVert}
\newcommand{\abs}[1]{\left\lvert #1 \right\rvert}
\newcommand{\opnorm}[1]{\norm{#1}_\text{op}}
\newcommand{\ve}{\varepsilon}
\newcommand{\coloneq}{\colonequals}

\DeclareMathOperator{\OST}{OST}
\DeclareMathOperator{\splitting}{SP}
\newcommand{\SP}[1]{\splitting ( #1 )}

\DeclareMathOperator*{\GL}{GL}
\DeclareMathOperator{\dom}{dom}

\DeclareMathOperator{\one}{\mathbf{1}}
\DeclareMathOperator{\an}{\mathbf{a}}
\DeclareMathOperator{\bn}{\mathbf{b}}

\DeclareMathOperator{\DiffGerm}{DiffGerm}
\DeclareMathOperator{\Germ}{Germ}
\DeclareMathOperator{\Gf}{Gf}

\newcommand{\BGp}{\ensuremath{G_{\mathrm{TM}}}}
\newcommand{\BGC}{\ensuremath{G_{\mathrm{TM}}^{\C}}}
\newcommand{\tBG}{\ensuremath{\bB_{\mathrm{TM}}}}
\newcommand{\tBGC}{\ensuremath{\tBG^\C}}
\newcommand{\ASUB}{\ensuremath{K_{n}}}
\newcommand{\cASUB}{\ASUB^{\C}}

\newcommand{\Mo}[1]{\ensuremath{E^{#1}}}
\newcommand{\Msp}[1]{\ensuremath{\Mo{#1}_0}}
\newcommand{\RT}{\ensuremath{\tT}}

\newcommand{\bB}{\ensuremath{\mathcal{B}}}

\newcommand{\fF}{\ensuremath{\mathcal{F}}}

\newcommand{\tT}{\ensuremath{\mathcal{T}}}
\newcommand{\pP}{\ensuremath{\mathcal{P}}}

\newcommand{\Lf}{\ensuremath{\mathbf{L}}}

\newcommand{\Frechet}{Fr\'echet }
\newcommand{\formalmor}{\mathrm{FB}}
\tikzstyle dtree=[grow'=up,sibling distance=2mm,level distance=2mm,thick]
\tikzstyle dtree node=[scale=0.3,shape=circle,very thin,draw]
\tikzstyle dtree black node=[style=dtree node,fill=black]
\newcommand{\onenode}{
  \begin{tikzpicture}[dtree]
    \node[dtree black node] {}
    ;
  \end{tikzpicture}
}

\title{The tame Butcher group}
 \author{G. Bogfjellmo\footnote{NTNU Trondheim, Norway \href{mailto:geir.bogfjellmo@math.ntnu.no}{geir.bogfjellmo@math.ntnu.no}}%
 \ \ and A.
Schmeding\footnote{NTNU Trondheim, Norway
\href{mailto:alexander.schmeding@math.ntnu.no}{alexander.schmeding@math.ntnu.no}
}}
\begin{document}

\maketitle

\begin{abstract}
 The Butcher group is a powerful tool to analyse integration methods for ordinary differential equations, in particular Runge--Kutta methods.
 Recently, a natural Lie group structure has been constructed for this group. Unfortunately, the associated topology is too coarse for some applications in numerical analysis.
 In the present paper, we propose to remedy this problem by replacing the Butcher group with the subgroup of all exponentially bounded elements.
 This ``tame Butcher group'' turns out to be an infinite-dimensional Lie group with respect to a finer topology.
 As a first application, we show that the correspondence of elements in the tame Butcher group with their associated B-series induces certain Lie group (anti)morphisms.
\end{abstract}

\medskip


\medskip

(secondary)

\tableofcontents

\section*{Introduction and statement of results}
 In the analysis of numerical integrators an important tool is the so called Butcher group.
 J.C.\ Butcher introduced this group in \cite{Butcher72} as a group of mappings on rooted trees.
 To each element in the group one associates a (formal) power series -- the B-series -- and the group product models composition of these power series.
 The interplay between the combinatorial structure of the trees and the group product enables one to handle (formal) power series solutions of non-linear ordinary differential equations.
 Building on Butcher's ideas, the Butcher group and methods associated to B-series have been extensively studied, see e.g.\ \cite{CHV2010,HLW2006,Brouder-04-BIT,MMMV14,BS14}.
 \medskip

 The Butcher group is in a certain sense very large, i.e.\ it contains elements whose associated B-series do not converge, even locally.
 To remedy this, the B-series are usually truncated and one has to derive some bound on the truncation error.
 In the literature, rigorous error estimates are available for elements in the group whose series coefficients satisfy exponential growth bounds, see e.g.\ the investigation in \cite{HL1997}.
 This motivates the study of the subgroup of all suitably bounded elements from the perspective of infinite-dimensional Lie theory.
 In the present paper, we investigate this subgroup, which we call the \emph{tame Butcher group}.

 Our investigation is based on a concept of $C^r$-maps between locally convex spaces known as Bastiani's calculus \cite{MR0177277} or Keller's $C^r_c$-theory~\cite{keller1974} (see \cite{milnor1983,hg2002a,neeb2006} for streamlined expositions).
 For the reader's convenience Appendix \ref{App: lcvx:diff} contains a brief recollection of the basic notions used throughout the paper.
 Our aim is to construct an infinite-dimensional Lie group structure on the tame Butcher group and to study its properties.
 It turns out that the tame Butcher group is a locally convex Lie group modelled on an inductive limit of Banach spaces.
 This structure is naturally related to the Lie group structure of the Butcher group constructed in \cite{BS14}.
 Moreover, as a consequence of the estimates established in \cite{HL1997}, one can even construct natural Lie group (anti)morphisms into groups of germs of diffeomorphisms.
 These morphisms represent the Lie theoretic formulation of the mechanism which associates to an element in the (tame) Butcher group its B-series.
 \medskip

 We now go into some more detail and explain the main results of the present paper.
 Let us first recall some notation and the definition of the (tame) Butcher group.
 Denote by $\RT$ the set of all rooted trees with a finite positive number of vertices.
 Furthermore, we let $\emptyset$ be the empty tree.
 Then the Butcher group is defined to be the set of all tree maps which map the empty tree to $1$, i.e.\
 \begin{displaymath}
  \BGp = \{a \colon \RT \cup \{ \emptyset \} \rightarrow \R \mid a (\emptyset) = 1\}.
 \end{displaymath}
 To define the group operation of the Butcher group, one interprets the values of a tree map as coefficients of a (formal) power series.
 Via this identification, the composition law for power series induces a group operation on $\BGp$. We refer to Section \ref{sect: tBG-LGP} for more details and the explicit formula.
 Note that the Butcher group contains arbitrary tree maps, i.e.\ there is no restriction on the value a tree map can attain at a given tree. \medskip

 The tame Butcher group $\tBG$ is the subgroup of $\BGp$ which consists of all elements $a$ which satisfy the following exponential growth condition
 \begin{displaymath}
  \exists C,K > 0 \text{ such that } \abs{a(\tau)} \leq CK^{|\tau|} \text{for all } \tau \in \RT \text{ where } |\tau| = \# \text{ nodes of } \tau.
 \end{displaymath}
 Note that the same constructions can be performed for tree maps with values in the field of complex numbers.
 As a result, one obtains the complex tame Butcher group $\tBGC$ which will be an important tool in our investigation.

 Returning to $\tBG$, we observe that the subspace topology induced by $\BGp$ on $\tBG$ is too coarse for our needs (see also the discussion in \cite[Remark 2.3]{BS14}).
 Hence we replace it with the inductive limit topology induced by a sequence of weighted mapping spaces.
 Their inductive limit turns out to be a Silva space (also known as a (DFS)-space, cf.\ Proposition \ref{prop: silva}) and the results in Section \ref{sect: LGP} subsume the following theorem.
 \medskip

 \textbf{Theorem A} \emph{The tame Butcher group $\tBG$ is a real analytic infinite-dimensional Lie group modelled on a Silva space.}\medskip

 At this point, it is worthwhile to stress that virtually every ``interesting'' element of the Butcher group is contained in the tame Butcher group.
 In particular, all elements of the Butcher group associated to commonly used numerical integrators, e.g.\ Runge--Kutta methods, are contained in $\tBG$.\footnote{This follows from \cite[Theorem 6.8]{Butcher72}, which asserts that Butcher's group $G_0$, which contains the Runge--Kutta methods, is contained in $\tBG$.}

 Having a Lie group structure on the tame Butcher group at our disposal, we then investigate the Lie theoretic properties of this group.
 An important observation will be that the Lie group structure on the tame Butcher group is -- albeit not directly induced from the Butcher group-- closely related to the Lie structure of the Butcher group.
 To make this explicit, consider the canonical inclusion $\alpha_\R \colon \tBG \rightarrow \BGp$.
 Then $\alpha_\R$ is a morphism of Lie groups and the induced morphism of Lie algebras $\Lf (\alpha_\R)$ separates the points.
 The morphisms $\alpha_\R$ and $\Lf (\alpha_\R)$ will be our main tools to transport structure to the tame Butcher group.
 In particular, these tools allow the Lie algebra of the tame Butcher group to be identified (see Section \ref{sect: BG-LAlg}):
 \medskip

 \textbf{The Lie Algebra of $\tBG$} \emph{The Lie algebra $\Lf (\tBG)$ of the tame Butcher group is isomorphic as an abstract Lie algebra to a Lie subalgebra of the Lie algebra $\Lf (\BGp)$ of the Butcher group.}
 \medskip

 Here the term ``abstract'' Lie subalgebra points to the the fact that $\Lf (\tBG)$ is not a topological Lie subalgebra of $\Lf (\BGp)$.
 However, as its natural Silva space topology is strictly finer than the subspace topology, the canonical inclusion becomes a morphism of topological Lie algebras.
 Moreover, the Lie bracket of the Lie algebra of the Butcher group restricts to the continuous Lie bracket of the Lie algebra of the tame Butcher group.

 Then we establish that the tame Butcher group is a regular Lie group.
 To understand these results first recall the notion of regularity for Lie groups.

  Let $G$ be a Lie group modelled on a locally convex space, with identity element $\one$, and
 $r\in \N_0\cup\{\infty\}$. We use the tangent map of the right translation
 $\rho_g\colon G\to G$, $x\mapsto xg$ by $g\in G$ to define
 $v.g\coloneq T_{\one} \rho_g(v) \in T_g G$ for $v\in T_{\one} (G) =: \Lf(G)$.
 Following \cite{HG15reg}, $G$ is called
 \emph{$C^r$-semiregular} if for each $C^r$-curve
 $\gamma\colon [0,1]\rightarrow \Lf(G)$ the initial value problem
 \begin{displaymath}
  \begin{cases}
   \eta'(t)&= \gamma(t).\eta(t)\\ \eta(0) &= \one
  \end{cases}
 \end{displaymath}
 has a (necessarily unique) $C^{r+1}$-solution
 $\Evol (\gamma)\coloneq\eta\colon [0,1]\rightarrow G$.
 Furthermore, if $G$ is $C^r$-semiregular and the map
 \begin{displaymath}
  \evol \colon C^r([0,1],\Lf(G))\rightarrow G,\quad \gamma\mapsto \Evol
  (\gamma)(1)
 \end{displaymath}
 is smooth, then $G$ is called \emph{$C^r$-regular}.
 If $G$ is $C^r$-regular and $r\leq s$, then $G$ is also
 $C^s$-regular. A $C^\infty$-regular Lie group $G$ is called \emph{regular}
 \emph{(in the sense of Milnor}) -- a property first defined in \cite{milnor1983}.
 Every finite dimensional Lie group is $C^0$-regular (cf. \cite{neeb2006}). Several
 important results in infinite-dimensional Lie theory are only available for
 regular Lie groups (see
 \cite{milnor1983,neeb2006,HG15reg}, cf.\ also \cite{KM97} and the references therein).

 Using recent results of Gl\"{o}ckner on regularity of Silva Lie groups (see \cite[Section 15]{HG15reg}), we can then prove the following theorem.
 \medskip

 \textbf{Theorem B} \emph{The tame Butcher group is $C^1$-regular and $C^0$-semiregular. The evolution map
  \begin{displaymath}
   \evol \colon C^1 ([0,1] , \tBG) \rightarrow \tBG ,\quad \eta \mapsto \Evol (\eta) (1)
  \end{displaymath}
 is even real analytic.
 }\medskip

Recall from \cite{milnor1983} that regular Lie groups admit a smooth Lie group exponential function.
For the (full) Butcher group $\BGp$ it is even known that the Lie group exponential map is a diffeomorphism, i.e.\ the Lie group $\BGp$ is exponential (cf.\ \cite[Section 5]{BS14}).
Contrary to the full Butcher group, the tame Butcher group is neither exponential nor even locally exponential.
In Section \ref{sect: nonexp}, we show that for each $\ve \neq 0$ the element of $\tBGC$ given by
\begin{displaymath}
   a_{\ve}(\tau) = \begin{cases}
   1, & \text{if $\tau= \emptyset$,} \\
   \ve,   & \text{if $\tau=\onenode$ (the one-node tree),} \\
   0 & \text{else,}
   \end{cases}
   \end{displaymath}
can not be contained in the image of the Lie group exponential map.
We note that the elements $a_\ve$ correspond to well known numerical integrators.
Namely, $a_1$ is the element associated to the forward Euler method, and in general $a_{\ve}$ corresponds to forward Euler with an adjusted step size.
As $\ve$ was arbitrary, this yields a curve through the identity in $\tBG$ whose intersection with the image of the exponential map is just the identity.
Hence $\tBG$ is not even locally exponential.
\medskip

From the point of view of numerical analysis, the failure of $\tBG$ to be exponential is reflected in the fact that the so-called ``modified vector fields'' of B-series methods are in general nonconvergent formal series that have to be truncated.
\medskip

In the last chapter we construct Lie group morphisms from the tame Butcher group to a class of Lie groups of germs of diffeomorphisms.
Consider a real Banach space $E$ and recall from \cite{MR2670675} that the group $\DiffGerm (\{(y_0,0)\}, E \times \R)$ of germs of real analytic diffeomorphisms which fix the point $(y_0,0)$ in $E \times \R$ is a Lie group.
For finite-dimensional $E$, the construction goes back to \cite{hg2007} (compare also \cite{MR1828283} and \cite{MR1292806}).
The Lie group morphisms envisaged here take  an element of the tame Butcher group to its the associated B-series with respect to a suitable analytic vector field.
As a consequence of estimates from backward error analysis in \cite{HL1997}, one obtains the following :\medskip

\textbf{Construction principle} \emph{Let $f \colon E \supseteq U \rightarrow E$ be a real analytic map defined on an open set of a Banach space $E$. Fix $y_0 \in U$.
Then the B-series of $a \in \tBG$ with respect to $f$ is defined as $B_f (a,y,h) \coloneq y + \sum_{\tau \in \RT} \frac{h^{|\tau|}a(\tau)}{\sigma (\tau)} F_f (\tau) (y)$ and}
  \begin{displaymath}
   B_f \colon \tBG \rightarrow \DiffGerm (\{(y_0,0)\}, E \times \R) ,\quad  a \mapsto [(y,h) \mapsto (B_f (a,y,h),h)]_\sim
  \end{displaymath}
 \emph{ is a real analytic Lie group antimorphism.} \medskip

Note that we do not obtain a morphism of Lie groups because our definition for multiplication follows the conventions in the literature (which constructs the product as composition in the opposite order).
The map $B_f$ can be interpreted as the map which associates to a numerical method the numerical solution of an ordinary differential equation, i.e.\ the germ represents a (numerical approximation to a) power series solution of the ordinary differential equation
  \begin{displaymath}(\star) \begin{cases}
              y'(t) &=f(y),\\
              y(0) &= y_0
     \end{cases}
 \end{displaymath}
 near $y_0$.
 Hence the map $B_f$ corresponds to the construction of the numerical solutions to $(\star)$.
 This construction incorporates information on the convergence of the solution and specialises the construction of the formal Butcher series (which can also be seen as a certain Lie group antimorphism, cf.\ Proposition \ref{prop: formalmor}).

In general, the behaviour of the antimorphism $B_f$ depends strongly on the choice of $f$. E.g.\ for the zero-map $0$ on some Banach space, $B_0$ is the trivial group morphism.
Moreover, the parameter $h$ introduced in the construction principle is usually interpreted as the step-size of the method in numerical analysis.
In particular, in concrete applications this parameter is assumed to be a positive number (which is often fixed).
However, to obtain a Lie group antimorphism, one cannot fix $h$ and has to allow $h$ to take more general values.

\section{Preliminaries on the (tame) Butcher group}\label{sect: tBG-LGP}

In this section we construct the tame Butcher group and discuss its topology.
The tame Butcher group is a subgroup of the Butcher group, thus we will first review the construction of the Butcher group.

\begin{nota}
 We write $\N \coloneq \{1,2,\ldots\}$, respectively $\N_0 \coloneq \N \cup \{0\}$. As usual $\R$ and $\C$ denote the fields of real and complex numbers, respectively.
 \end{nota}

\section*{The Butcher group}

We recommend \cite{HLW2006,CHV2010} for an overview of basic results and algebraic properties of the Butcher group.
Let us now establish some notation.

 \begin{nota}
 \begin{enumerate}
  \item A \emph{rooted tree} is a connected \emph{finite} graph without cycles with a distinguished node called the \emph{root}.
  Identify rooted trees if they are graph isomorphic via a root preserving isomorphism.

  Let $\RT$ be \emph{the set of all rooted trees} and write $\RT_0 := \RT \cup \{\emptyset\}$ where $\emptyset$ denotes the empty tree.
  The \emph{order} $|\tau|$ of a tree $\tau \in \RT_0$ is its number of vertices.
  \item An \emph{ordered subtree}\footnote{The term ``ordered'' refers to that the subtree remembers from which part of the tree it was cut.} of $\tau \in \RT_0$ is a subset $s$ of all vertices of $\tau$ which satisfies
    \begin{compactitem}
     \item[(i)] $s$ is connected by edges of the tree $\tau$,
     \item[(ii)] if $s$ is non-empty, it contains the root of $\tau$.
    \end{compactitem}
   The set of all ordered subtrees of $\tau$ is denoted by $\OST (\tau)$.
   Further, $s_\tau$ denotes the tree given by vertices of $s$ with root and edges induced by $\tau$.
  \item A \emph{partition} $p$ of a tree $\tau \in \RT_0$ is a subset of edges of the tree.
 We denote by $\pP (\tau)$ the set of all partitions of $\tau$ (including the empty partition).
 \end{enumerate}\label{nota: OST:P}
  Associated to $s \in \OST (\tau)$ is a forest $\tau \setminus s$ (collection of rooted trees) obtained from $\tau$ by removing the subtree $s$ and its adjacent edges.
  Similarly, to a partition $p \in \pP (\tau)$ a forest $\tau \setminus p$ is associated as the forest that remains when the edges of $p$ are removed from the tree $\tau$.
  In either case, we let $\# \tau \setminus p$ be the number of trees in the forest.
\end{nota}

\begin{defn}[Butcher group]
 Define the \emph{complex Butcher group} as the set of all tree maps
  \begin{displaymath}
   \BGC = \{a \colon \RT_0 \rightarrow \C \mid a(\emptyset) = 1\}
  \end{displaymath}
 together with the group multiplication
 \begin{equation}\label{eq: BGP:mult}
  a\cdot b (\tau) \coloneq \sum_{s \in \OST (\tau)} b(s_\tau) a(\tau \setminus s) \quad \text{with} \quad a(\tau \setminus s) \coloneq \prod_{\theta \in \tau \setminus s} a (\theta)
 \end{equation}
 and inversion (cf.\  \cite{CHV2010})
 \begin{equation}\label{eq: BGP:invers}
   a^{-1} (\tau) = \sum_{p \in \pP (\tau)} (-1)^{\# \tau \setminus p} a (\tau \setminus p) \quad \text{with } a(\tau \setminus p) \coloneq \prod_{\theta \in \tau \setminus p} a (\theta).
  \end{equation}
 The identity element $e \in \BGp$ with respect to this group structure is
  \begin{displaymath}
   e \colon \RT_0 \rightarrow \C, \ e (\emptyset) = 1 , \ e(\tau) = 0, \ \forall \tau \in \RT.
  \end{displaymath}
  We define the \emph{(real) Butcher group} as the real subgroup
  \begin{displaymath}
   \BGp = \{a \in \BGC \mid \im \, a \subseteq \R\}
  \end{displaymath}
  of $\BGC$.
  As the real Butcher group is referred to in the literature as ``the Butcher group'', we shall always denote by ``the Butcher group'' the real Butcher group.
 \end{defn}

\begin{rem}\label{rem: pwtop}
 The vector space of all tree maps $\K^{\RT_0}$ becomes a locally convex space with the \emph{topology of pointwise convergence}.
 This topology is induced by the semi-norms
  \begin{displaymath}
   p_\tau \colon \K^{\RT_0} \rightarrow \K , a \mapsto |a(\tau)|, \quad \text{for } \tau \in \RT_0.
  \end{displaymath}
 Recall that a mapping $f \colon X \rightarrow \K^{\RT_0}$ is continuous if and only if $p_\tau \circ f$ is continuous for all $\tau \in \RT_0$, i.e.\ if and only if the composition of $f$ with each evaluation map
  \begin{displaymath}
   \ev_\tau \colon \K^{\RT_0} \rightarrow \K , a \mapsto a(\tau)
  \end{displaymath}
 is continuous. Furthermore, we remark that the (complex) Butcher group is the closed affine subspace $e + \K^{\RT}$ of $\K^{\RT_0}$ (where $\K^{\RT}$ has been identified with $\{a \in \K^{\RT_0} \mid a(\emptyset)=0\}$).
\end{rem}

\begin{prop}[{{\cite[Theorem 2.1]{BS14}}}]\label{prop: BGP:Lie}
 Identify the (complex) Butcher group with the affine subspace $e +\K^{\RT}$ of $\K^{\RT_0}$ with the topology of pointwise convergence.
 With respect to this topology and submanifold structure, the (complex) Butcher group turns into a $\K$-analytic Baker--Campbell--Hausdorff-Lie group.
\end{prop}

 \section*{The tame Butcher group}
Elements in the Butcher group $\BGp$ are closely related to a certain kind of formal power series, the so-called Butcher series (or B-series for short).
We will now briefly recall the necessary concepts from the theory of Butcher series (cf.\ \cite[III.1]{HLW2006}).

\begin{defn}[Elementary differentials and B-series]
\begin{enumerate}
 \item Consider a finite family of trees $\tau_1 , \tau_2, \ldots , \tau_m \in \RT$.
 We define a new tree $\tau \coloneq [\tau_1, \ldots , \tau_m]$ by grafting the roots of the trees $\tau_i$ to a common new vertex which we define as the root of $\tau$.
 \item Define recursively the \emph{symmetry coefficient} $\sigma (\tau)$ of a tree $\tau \in \RT$ via $\sigma (\bullet) =1$ and, for $\tau = [\tau_1, \ldots, \tau_m]$
  \begin{displaymath}
   \sigma (\tau) = \sigma (\tau_1) \cdot \ldots \cdot \sigma (\tau_m) \cdot \mu_1!\mu_2! \ldots ,
  \end{displaymath}
 where the integers $\mu_1, \ldots , \mu_m$ count equal trees among $\tau_1, \ldots , \tau_m$.
 \item Let now $f \colon E \supseteq U \rightarrow E$ be a smooth mapping defined on an open subset of the normed space $(E,\norm{\cdot})$. Then we define for $\tau \in \RT$ recursively the \emph{elementary differential} $F_f (\tau) \colon U \rightarrow E$ via
 $F_f(\bullet) (y) = f(y)$ and
  \begin{displaymath}
   F_f(\tau)(y) \coloneq f^{(m)} (y) (F_f (\tau_1)(y), \ldots F_f (\tau_m)(y)) \quad \text{ for } \tau = [\tau_1, \ldots , \tau_m]
  \end{displaymath}
 where $f^{(m)}$ denotes the $m$th \Frechet derivative.\footnote{Recall from \cite{MR1853240} that smooth mappings in our sense on normed spaces are smooth in the sense of \Frechet-differentiability.}

 Now we can define for $a \in \BGp$, $y \in U$ and $h \in \R$ a formal series
  \begin{displaymath}
   B_f (a,y,h) \coloneq y + \sum_{\tau \in \RT} \frac{h^{|\tau|}}{\sigma (\tau)}a(\tau) F_f (\tau)(y),
  \end{displaymath}
 called \emph{B-series}.
\end{enumerate}\label{defn: Bseries}
Note that usually the map $f$ is fixed in advance and thus the dependence on $f$ is suppressed in the notation of elementary differentials and in the B-series.
\end{defn}

\begin{rem}
 In the numerical analysis literature the parameter $h$ is interpreted as step-size of a numerical scheme.
 Hence it is usually assumed to be positive.
 Due to technical reasons which will become apparent later, we allow arbitrary $h \in \K$.
\end{rem}

In application one is interested in those B-series which converge at least locally around a given point.
However, the B-series of many elements in the Butcher group do not converge locally.
With respect to the Lie group topology on $\BGp$ (see Proposition \ref{prop: BGP:Lie}) these elements are even dense in $\BGp$.
If one restricts attention to elements which satisfy a certain exponential growth restriction, one can force the B-series to converge at least for analytic vector fields.

\begin{prop}\label{prop: convergence}
Fix a complex analytic map $f\colon E \supseteq U \to E$ on an open $y_0$-neighbourhood in a complex Banach space $(E,\norm{\cdot})$.
Let $a$ be an element of $\BGp$ such that there are $C,K >0$ with $|a(\tau)|\le CK^{|\tau|}$ for all $\tau \in \RT$.
Then there is a $y_0$-neighbourhood $V \subseteq U$ and a constant $h_0 \coloneq h_0 (f,V,K) >0$ such that the B-series
\begin{equation}
\label{eq: Bseries}
 B_f (a,y,h) \coloneq  y+\sum_{\tau \in \RT} \frac{h^{|\tau|}a(\tau)}{\sigma(\tau)} F_f(\tau)(y)
\end{equation}
converges for all $\abs{h} \leq h_0$ and $y \in V$.
\end{prop}

\begin{proof}
 Choose $R > 0$ small enough, such that the ball $B_R^E (y)$ is contained in $U$ and there is a constant $M>0$ with $\sup_{z \in \overline{B_R^E (y)}}\norm{f(z)} \leq M$.
 This is possible by continuity of $f$.
 Now we want to apply \cite[Lemma 1.5]{MR2670675} to obtain a Cauchy estimate for the \Frechet derivatives of $f$ on $B_{\frac{R}{2}}^E (y_0)$.
 Since $f$ is bounded by $M$ on the $R$-ball around $y_0$ we deduce from the proof of loc.cit. (see \cite[p.118]{MR2670675}) for all $y \in B_{\frac{R}{2}}^E (y_0)$ the estimates
 \begin{equation}\label{est: Cauchy}
  \opnorm{f^{(k)} (y)} \leq k! M \left(\frac{e}{R}\right)^k \quad \text{for all } k \in \N_0 \text{ where } e = 2.718281828\ldots
 \end{equation}
 We can now argue as in \cite[Lemma 9]{HL1997}. Under equivalent assumptions as ours\footnote{loc.cit. has the condition $|a(\tau)|\le \gamma(\tau) \mu \kappa^{|\tau|-1}$. The factor $\gamma(\tau)$ is due to a different scaling of the tree coefficients.}, the estimate \eqref{est: Cauchy} yields the bound
 \begin{equation}\label{est: Bseries}
  \norm{\sum_{\substack{\tau \in \RT \\ |\tau| = q}} \frac{h^{|\tau|}a(\tau)}{\sigma(\tau)} F_f(\tau)(y)} \le \frac 12 C \left(\frac{4eMK\abs{h}}{R}\right)^q.
 \end{equation}
 Note that in \cite{HL1997} only the case of a finite-dimensional space is considered and the computations are carried out with respect to the $\ell^1$-norm.
 However, the necessary estimate from the proof of \cite[Lemma 9]{HL1997} carry over verbatim to our setting if we use \eqref{est: Cauchy}.
 Now we fix arbitrary $0 <  h_0 (f,V,K) < \frac{R}{4eMK}$. Then \eqref{est: Bseries} ensures convergence of \eqref{eq: Bseries} when $y$ is contained in $V \coloneq B_{\frac{R}{2}}^E (0)$ and $\abs{h}\leq h_0$ .
\end{proof}

We single out all elements of the Butcher group which satisfy an exponential growth restrictions to construct a Lie group which is amenable for applications in numerical analysis.

\begin{defn}
 The \emph{(complex) tame Butcher group} $\tBGC$ is the subgroup of all elements in $a \in \BGC$ which satisfy the exponential growth restriction
 \begin{displaymath}
 \quad \exists C , K >0 \text{ such that for all } \tau \in \RT \quad |a(\tau)| \leq CK^{|\tau|}
 \end{displaymath}
 (we will see in Lemma \ref{lem: tame:subgp} below that $\tBGC$ is indeed a subgroup of $\BGC$).

 Moreover, define the \emph{(real) tame Butcher group} $\tBG$ via $\tBG = \tBGC \cap \BGp$ (which is a subgroup of $\BGp$ by the above).
 In the following, we will always write ``the tame Butcher group'' for the real tame Butcher group.
\end{defn}

\begin{lem}\label{lem: tame:subgp}
 The group operations of $\BGC$ turn $\tBGC$ into a subgroup. Thus also $\tBG$ becomes a subgroup of $\BGp$.
\end{lem}

\begin{proof}
 Let $a,b$ be elements in $\tBGC$ and choose constants $C_a,C_b, K>0$ such that $|a(\tau)|\leq C_a K^{|\tau|}$ and $|b(\tau)|\leq C_b K^{|\tau|}$ hold for all $\tau \in \RT$.
 Without loss of generality, we may assume that $C_a,C_b \geq 1$.
 Now we compute exponential growth bounds for $a \cdot b$ and $a^{-1}$.
 Fix $\tau \in \RT$ and consider the product of $a$ and $b$ evaluated in $\tau$.
 \begin{equation} \begin{aligned}
  |a\cdot b (\tau)| &\stackrel{\hphantom{\eqref{eq: OST:est}}}{=} \left| \sum_{s \in \OST(\tau)} b(s_\tau) a(\tau\setminus s) \right| \leq \sum_{s \in \OST (\tau)} |b(s_\tau)| \prod_{\theta \in \tau \setminus s} |a(\theta)| \\
		    &\stackrel{\hphantom{\eqref{eq: OST:est}}}{\leq} C_b K^{|s_\tau|} \prod_{\theta \in \tau \setminus s} C_a K^{|\theta|} = \sum_{s \in \OST(\tau)} K^{|\tau|} C_b \prod_{\theta \in \tau \setminus s} C_a  \\
		    &\stackrel{\eqref{eq: OST:est}}{\leq} \sum_{s \in \OST(\tau)} K^{|\tau|} C_b (C_a)^{|\tau|} \stackrel{\eqref{eq: OST:est}}{\leq} 2^{|\tau|} C_b (C_a K)^{|\tau|} = C_b (2C_aK)^{|\tau|}
		    \end{aligned} \label{eq: prod:est}
 \end{equation}
 Thus $a \cdot b$ is again exponentially bounded, whence $a \cdot b \in \tBGC$ for all $a,b \in \tBGC$.

 Now consider the inverse evaluated at $\tau$.
 \begin{equation} \begin{aligned}
  |a^{-1} (\tau)| &= \left| \sum_{p \in \pP (\tau)} (-1)^{\# \tau \setminus p} a(\tau \setminus p)\right| \leq \sum_{p \in \pP (\tau)} \prod_{\theta \in \tau \setminus p} |a(\theta)| \\
		  &\leq \sum_{p \in \pP (\tau)} \prod_{\theta \in \tau \setminus p} C_a K^{|\theta|} \stackrel{\eqref{eq: Part:est}}{\leq} 2^{|\tau|} (C_a K)^{|\tau|} = (2C_aK)^{|\tau|}
		  \end{aligned} \label{eq: inv:est}
 \end{equation}
 Hence, $a^{-1}$ is exponentially bounded and thus $a^{-1} \in \tBGC$ if and only if $a \in \tBGC$.

 The assertion concerning the tame Butcher group now immediately follows from the definition of $\tBG$.
 \end{proof}

Our aim is now to construct a model space for the group $\tBGC$ from an inductive limit of weighted function spaces.

\begin{defn}
 Fix $\K \in \{\R,\C\}$ and denote by $\abs{\cdot}$ the usual absolute value on $\K$.
 For $\omega \colon \RT_0 \rightarrow \K$ we define the $\K$-vector space
 \begin{displaymath}
  \K^{\RT_0} (\omega) \coloneq \left\{a \in \K^{\RT_0} \middle| \sup_{\tau \in \RT_0 } |a (\tau) \omega (\tau))| < \infty\right\}.
 \end{displaymath}
We call $\omega$ a \emph{weight} and $\K^{\RT_0}(\omega)$ a \emph{weighted mapping space}.
\end{defn}

\begin{lem}\begin{enumerate}
             \item For a weight $\omega$, the weighted mapping space $\K^{\RT_0} (\omega)$ is a Banach space over $\K$ with respect to the norm
 \begin{displaymath}
  \norm{a}_{\omega} \coloneq \sup_{\tau \in \RT_0} |a (\tau) \omega (\tau)|.
 \end{displaymath}
            \end{enumerate}
Let $\omega , \omega' \colon \RT_0 \rightarrow [0,\infty[$ be weights which satisfy $\omega (\tau) \geq \omega'(\tau)$ for all $\tau \in \RT_0$.
Then
 \begin{enumerate}
  \item[{\upshape (b)}] $\norm{a}_{\omega} \leq \norm{a}_{\omega'}$ for all $a \in \K^{\RT_0} (\omega)$.
  Hence $\K^{\RT_0} (\omega) \subseteq \K^{\RT_0} (\omega')$ and the inclusion is even continuous.
  \item[{\upshape (c)}] If in addition $\displaystyle\lim_{|\tau| \rightarrow \infty} \frac{\omega' (\tau)}{\omega (\tau)} = 0$ holds, then the inclusion $\K^{\RT_0} (\omega) \subseteq \K^{\RT_0} (\omega')$ is a compact operator.
 \end{enumerate}\label{lem: sq:silva}
\end{lem}

\begin{proof}
 \begin{enumerate}
  \item Recall that $\K^{\RT_0} =\prod_{\tau \in \RT_0} \K , a \mapsto (a(\tau))_{\tau \in \RT_0}$.
  Moreover, the set $\RT_0$ is countable, whence we can define the weighted sequence space
    \begin{displaymath}
     \ell^{\infty}_\K (\omega) \coloneq \left\{(a_\tau)_{\tau \in \RT_0} \in \prod_{\tau \in \RT_0} \K \ \middle| \ \sup_{\tau \in \RT_0} |a_{\tau}\omega (\tau)| < \infty\right\}.
    \end{displaymath}
  By construction we have $\K^{\RT_0} (\omega) = \ell^\infty_\K (\omega)$.
  Now $\ell^{\infty}_\K (\omega)$ is a Banach space (see \cite[\S 4 5.4]{FW68}) with respect to the norm $\norm{(a_\tau)_\tau} = \sup_{\tau \in \RT_0} \abs{a_\tau \omega (\tau)}$.
  We conclude that $\K^{\RT_0} (\omega)$ is a Banach space.
  \item[(b) and (c)] Assertion (b) follows from \cite[\S 9 4.1]{FW68} and (c) is \cite[\S 20 4.2 Satz]{FW68}. \qedhere
 \end{enumerate}
\end{proof}

\begin{setup}[Weights for the (complex) tame Butcher group]\label{setup: model:space}
 For $k \in \N$ we define the weight $\omega_k \colon \RT_0 \rightarrow [0,\infty[ , \tau \mapsto \left(\tfrac{1}{2^k}\right)^{|\tau|}$.
 By construction we have for each $k \in \N$ the relations
  \begin{displaymath}
   \omega_{k+1} (\tau ) \leq \omega_{k} (\tau) \text{ for all } \tau \in \RT_0 \text{ and } \lim_{|\tau| \rightarrow \infty} \frac{\omega_{k+1} (\tau)}{\omega_{k} (\tau)} = 0.
  \end{displaymath}
 Thus by Lemma \ref{lem: sq:silva} we obtain ascending sequences of Banach spaces
  \begin{displaymath}
   \K^{\RT_0} (\omega_1) \subseteq \K^{\RT_0} (\omega_2) \subseteq \K^{\RT_0} (\omega_3) \subseteq \cdots
  \end{displaymath}
 such that the inclusion maps are compact operators.
 Hence the (locally convex) inductive limit $\Mo{\K} = \bigcup_{k \in \N} \K^{\RT_0} (\omega_k)$ of this sequence is a Silva space\footnote{Recall that by definition, a Silva space is the inductive limit of a sequence of Banach spaces such that the bonding maps in this sequence are compact operators. Here the bonding maps are just the canonical inclusions. By \cite[Proposition 4.4. (a)]{hg2011} (or \cite{MR0287271}) Silva spaces are complete and have the Hausdorff property.}.
\end{setup}

\begin{rem} \label{rem: weights:mult}
 For $k \in \N$ the weight $\omega_k$ is multiplicative for tree partitions, i.e.\ if we partition a tree $\tau$ into a forest $\fF$ of ordered subtrees this entails $\prod_{\theta \in \fF} \omega_k (\theta) = \omega_k (\tau)$.
\end{rem}

\begin{lem}\label{lem: cI:cont}
 The canonical inclusion $I_\K \colon \Mo{\K} \rightarrow \K^{\RT_0}$ of the Silva space into $\K^{\RT_0}$ with the topology of pointwise convergence is continuous linear.
\end{lem}

\begin{proof}
 By Remark \ref{rem: pwtop} it suffices to prove that for each $\tau \in \RT_0$ the map $\ev_\tau \circ I_\K \colon \Mo{\K} \rightarrow \K , a \mapsto a(\tau)$ is continuous.
 To see this note that by definition of a Silva space $\ev_\tau \circ I_\K$ will be continuous if $\ev_\tau \circ I_\K|_{\K^{\RT_0} (\omega_k)}$ is continuous for each $k \in \N$.
 Thus for $a \in \K^{\RT_0} (\omega_k)$ we derive from $|\ev_\tau I_\K (a)| = | a(\tau)| \leq \tfrac{1}{\omega_k (\tau)}\norm{a}_{\omega_k}$ that the map is indeed continuous.
\end{proof}

\begin{cor}\label{cor: ev:cont}
 The point evaluations $\ev_{\tau} \colon \Mo{\K} \rightarrow \K, a \mapsto a(\tau)$ are continuous linear for each $\tau \in \RT_0$.
 Moreover, the functionals $\ev_\tau , \tau \in \RT_0$ separate the points in $\Mo{\K}$ and restrict to continuous linear maps on every step of the inductive limit.
\end{cor}

We are now in a position to define the model space of the (complex) tame Butcher group.
To this end consider the subspace
\begin{displaymath}
   \Msp{\K} \coloneq \{a \in \Mo{\K} \mid a (\emptyset) = 0\}
  \end{displaymath}
of $\Mo{\K}$ (see \ref{setup: model:space}).
By Lemma \ref{lem: cI:cont} the subspaces $\Msp{\K}$ is closed as the preimage of $(\ev_\emptyset \circ I_\K)^{-1} (1)$.

\begin{lem}\label{lem: aff:subs}
The (complex) tame Butcher group can be identified with the affine closed subspace $e + \Msp{\K}$ of $\Mo{\K} = \bigcup_{k \in \N} \K^{\RT_0} (\omega_k)$ (where $e$ is the unit element in the (complex) tame Butcher group). Further, the subspace $\Msp{\K}$ is also a Silva space.
\end{lem}

\begin{proof}
For $b$ in the (complex) Butcher group, consider $C,K >0$ such that for all $\tau \in \RT$ we have $|b(\tau)| < CK^{|\tau|}$.
By choice of the weights in \ref{setup: model:space}, there is $N \in \N$ with $\omega_N (\tau) \leq \frac{1}{K^{|\tau|}}$ for all $\tau \in \RT$.
Hence, $\sup_{\tau \in \RT_0} \abs{b(\tau) \omega_N (\tau)} < C < \infty$, i.e.\ $b \in \K^{\RT_0} (\omega_N)$.
Moreover, as $a(\emptyset) = 1$ holds for all $a \in \tBGC$, the (complex) Butcher group coincides with the affine subspace $\{a \in \Mo{\K} \mid a(\emptyset) = 1\} = e + \Msp{\K}$.
To see that $\Msp{\K}$ is a Silva space define $L_k \coloneq \{a \in \K^{\RT_0} (\omega_k) \mid a(\emptyset)=0\}$. Then clearly $ \K^{\RT_0} (\omega_k)  = L_k \times \K$ and we have $\Mo{\K} = \lim (L_k \times \K)  = \lim L_k \times \K =\Msp{K} \times \K$ (the inductive limit commutes with finite direct products by \cite[Theorem 3.4]{MR1878717}). In particular, the closed subspace $\Msp{\K}$ is a Silva space as the inductive limit of the $L_k$.
\end{proof}

\begin{rem}\label{rem: SILVA}
 The model space of the (complex) Butcher group is an inductive limit of Banach spaces, even more special: a Silva space.
 For these spaces it is known (see Proposition \ref{prop: silva}) that mappings are continuous (or differentiable) if and only if they are continuous (or differentiable) on each step of the inductive limit.
 Hence the leading idea will be to compute differentiability and continuity of mappings on the steps of the inductive limit.
 We will see later that the problem at hand can even be reduced to consider only mappings between Banach spaces.
\end{rem}

In the following sections we will mostly be concerned with the complex tame Butcher group.
If we can establish the Lie group structure for this group then the corresponding results for the tame Butcher group will follow in the wash by a complexification argument.
We prepare this by showing that the complex model space of $\tBGC$ is a complexification of the model space of $\tBG$.

\begin{lem}\label{lem: compl:Banach}
 Let $\omega \colon \N \rightarrow \R$ be a weight, then the Banach space $(\C^{\RT_0} (\omega), \norm{\cdot}_{\omega})$ is a complexification of the real Banach space $(\R^{\RT_0} (\omega), \norm{\cdot}_{\omega})$.
\end{lem}

\begin{proof}
 In the following we denote by ``$i$'' the imaginary unit.
 A tree map $c \in \C^{\RT_0}$ splits canonically into a sum $c = a + i b$ where $a,b \in \R^{\RT_0}$.
 As $\omega$ takes only real values, the norm yields $\norm{a+ ib}_{\omega}= \sup_{\tau \in \RT_0} \sqrt{(a (\tau)\omega(\tau ))^2 + (b(\tau)\omega (\tau))^2}$.
 Using this identity an easy computation yields the estimates
 \begin{align}
  \max \{\norm{a}_{\omega}, \norm{b}_{\omega}\} &\leq \norm{a+ ib}_{\omega} \leq \sqrt{2}  \max \{\norm{a}_{\omega}, \norm{b}_{\omega} \}.\label{eq: norm1}
 \end{align}
 Note the estimate becomes an equality if either the imaginary or real part of $c = a + ib$ vanishes.
 Thus component wise splitting of elements in $\C^{\RT_0}(\omega)$ into real and imaginary part allows us to identify $\C^{\RT_0} (\omega) = \R^{\RT_0}(\omega) \oplus i \R^{\RT_0} (\omega)$ as locally convex spaces.
\end{proof}

\begin{prop}\label{prop: complexification}
 The Silva space $\Msp{\C}$ is a complexification of the Silva space $\Msp{\R}$.
\end{prop}

\begin{proof}
 Clearly as vector spaces, we have $(\Msp{\R})_\C = \Msp{\R} \otimes i \Msp{\R} = \Msp{\C}$ as a real locally convex space.
 As inductive limits commute with finite products by  \cite[Theorem 3.4]{MR1878717}, we see that  $(\Msp{\R})_\C = \Msp{\C}$ also as a complex locally convex space.
\end{proof}

\section{The Lie group structure of the tame Butcher group}\label{sect: LGP}

In this section we will prove that the tame Butcher group is a Lie group modelled on a Silva space.
We will deal with the complex case first and exploit the properties of Silva spaces recorded in Remark \ref{rem: SILVA}.

The group operations of $\tBGC$ naturally restrict to mappings from the steps of the inductive limit into $\Mo{\C}$.
Using suitable estimates (cf.\ Appendix \ref{App: est:TM}) these mappings restrict even to maps from open zero-neighbourhoods in a step into a suitable step of the inductive limit.

\begin{nota}\label{nota: ball}
 Recall that $e$ is the unit element in $\Mo{\C}$ and $e \in \C^{\RT_0} (\omega_k) \subseteq \Mo{\C}$ (note that the Banach space does not carry the subspace topology from $\Mo{\C}$).
 We denote by $B_R^{\omega_k} (e)$ the $R$-ball around $e$ in the Banach space $\C^{\RT_0} (\omega_k)$.
\end{nota}

\begin{setup}\label{setup: loc:operat}
 Fix $k \in \N$ and $R \geq 1$.
 Then define the neighbourhood $W(k,R) \coloneq B_R^{\omega_k} (e) \cap \tBGC$.
 Using the estimates \eqref{eq: prod:est} and \eqref{eq: inv:est}, there is an $N \coloneq N(k,R) \in \N$ such that the group operations of $\tBGC$ factor through maps
  \begin{align*}
   \mu_k^R \colon W(k,R) \times W(k,R) &\rightarrow \C^{\RT_0} (\omega_N) \cap \tBGC,\quad \mu_k^R (a,b) (\tau) = \sum_{s \in \OST (\tau)} b(s_\tau) a(\tau \setminus s)\\
   \text{and } \iota_k^R \colon W(k,R) &\rightarrow \C^{\RT_0} (\omega_N) \cap \tBGC,\quad \iota_k^R (a) (\tau) =\sum_{p \in \pP (\tau)} (-1)^{\#\tau \setminus p} a (\tau \setminus p).
  \end{align*}
  Enlarging $N = N(k,R)$, we can from now on assume that $\frac{8R}{2^{N-k}} < 1$.
  This entails
  \begin{equation}\label{eq: norm:est}
   \frac{\omega_N (\tau)}{\omega_k (\tau)} (8R)^{|\tau|} = \left(\frac{8R}{2^{N-k}}\right)^{|\tau|} < 1, \quad \forall k \in \N.
  \end{equation}
\end{setup}

 To turn $\tBGC$ into a Lie group, we consider continuity and differentiability of the restricted operations from \ref{setup: loc:operat}.
 By virtue of Proposition \ref{prop: silva} this will yield smoothness of the group operations.
 As the ambient spaces of our mappings become (affine) subspaces of Banach spaces this simplifies the problem enormously.
 Let us prove first a preliminary proposition.

\begin{prop}\label{prop: ops:cont}
 For each $k \in \N$ and each $R \geq 1$ the maps $\mu_k^R$ and $\iota_k^R$ are continuous.
\end{prop}

\begin{proof}
 Let $U$ be an open subset of $\C^{\RT_0} (\omega_N)$.
 We will prove that $(\mu^k_R)^{-1} (U)$ or $(\iota^k_R)^{-1} (U)$ are open by showing that these sets are neighbourhoods of all of their points.
 \medskip

 \textbf{1. $\mu_k^R$ is continuous} Choose $(a,b) \in (\mu^R_k)^{-1} (U)$.
 Let us construct $1> \ve >0$ such that $\mu_k^R$ maps the set $B_\ve^{\omega_k} (a) \times  B_\ve^{\omega_k} (b)$ to $U$.
 Assume that $\norm{a}_{\omega_k} + \norm{b}_{\omega_k} + 2\ve < R$.
 Then for arbitrary $c, d \in B_\ve^{\omega_k} (0) \cap \Msp{\C}$ the estimate from Lemma \ref{lem: estimates1} is applicable:
  \begin{align*}
   \norm{\mu_k^R (a+c, b+d) - \mu_k^R(a,b)}_{\omega_k} &= \norm{(a + c) \cdot (b + d) - a \cdot b}_{\omega_N} \\
										    &\leq \sup_{\tau \in \RT_0} 2\ve \frac{(4R)^{|\tau|}}{\omega_k (\tau)} \omega_N (\tau) \stackrel{\eqref{eq: norm:est}}{\leq} 2\ve .
  \end{align*}
 Hence for small enough $\ve$, $\mu_k^R$ maps $(B_\ve^{\omega_k} (a) \times  B_\ve^{\omega_k} (b)) \cap (\tBGC \times \tBGC)$ to $U$.
 In conclusion, $\mu_k$ is continuous.
 \medskip

 \textbf{2. $i_k^R$ is continuous} Consider $a \in (i_k^R)^{-1} (U)$.
 We construct $1> \delta > 0$ such that $i_k^R (B_\delta^{\omega_k} (a) \cap \tBGC) \subseteq U$.
 Let $c$ be in $B_\delta^{\omega_k}(0) \cap \Msp{\C}$.
 From \eqref{setup: loc:operat} we deduce
  \begin{align*}
   \norm{i_k^R (a + c) - i_k^R (q)}_{\omega_N} &= \sup_{\tau \in \RT_0} \abs{\sum_{p\in \pP (\tau)} \hspace{-.15cm} (-1)^{\# \tau \setminus p}((a + c) (\tau \setminus p) - a(\tau \setminus p)} \omega_N (\tau)\\
								  &\hspace{-.55cm}\stackrel{\text{Lemma  \ref{lem: tree:diff:est}}}{\leq} \sup_{\tau \in \RT} \sum_{p\in \pP (\tau)} \frac{2^{|\tau|}(2R)^{|\tau|-1} \delta}{\prod_{\theta \in \tau \setminus p}\omega_k (\theta)} \omega_N (\tau)\\
								  &\hspace{-3pt}\stackrel{\eqref{eq: Part:est}}{\leq} \sup_{\tau \in \RT} \frac{4^{|\tau|}(2R)^{|\tau|-1} \delta}{\prod_{\theta \in \tau \setminus s}\omega_k (\theta)} \omega_N (\tau) \leq \sup_{\tau \in \RT} \frac{(8R)^{|\tau_n|} \delta}{\omega_k (\tau)} \omega_N (\tau) \stackrel{\eqref{eq: norm:est}}{\leq} \delta
  \end{align*}
 Again for $\delta$ small enough, $i_k^R$ maps $B_\delta^{\omega_k}(a) \cap \tBGC$ to $U$.
 Thus $i_k^R$ is continuous.
\end{proof}

Before continuing, we urge the reader to review Appendix \ref{App: lcvx:diff} for basic facts on infinite dimensional calculus and locally convex manifolds which will be needed in a moment.

\begin{thm}\label{thm: tameBGp:Lie}
  The subspace topology induced by $\Mo{\C}$ on $\tBGC$ turns $\tBGC$ into a (complex analytic) Lie group modelled on the Silva space $\Msp{\C}$.
 \end{thm}

 \begin{proof}
 Fix $R \geq 1$ and consider for $k \in \N$ the maps $\mu_k^R$, $\iota_k^R$ and choose $N = N(k,R)$ as in \ref{setup: loc:operat}.
 Define $L_N \coloneq \{a \in \C^{\RT_0} (\omega_N) \mid a (\emptyset) = 0\}$ and $\Psi \colon \tBGC \rightarrow \Mo{\C}, a \mapsto a - e$ (where $e \in \tBGC$ is the identity element).
 Then Proposition \ref{prop: ops:cont} shows that $\Psi \circ \mu_k^R$ and $\Psi \circ \iota_k^R$ are continuous maps from an open set of a locally convex space into the complete locally convex space $L_N$.

 Recall from Corollary \ref{cor: ev:cont} that the point evaluation mappings $\ev_\tau , \tau \in \RT_0$ form a set of continuous linear mappings on $L_N$.
 Further, these functionals clearly separate the points on $L_N$.
 From the formula for $\mu_k^R$ and $\iota_k^R$ in \ref{setup: loc:operat} and the definition of $\Psi$ and $\ev_\tau$ it is obvious that for each $\tau \in \RT_0$ the maps $\ev_\tau \circ \Psi \circ \mu_k^R$ and $\ev_\tau \circ \Psi \circ \iota_k^R$ are continuous polynomials and therefore complex analytic.
 Now Lemma \ref{lem:folklore} implies that $\mu_k^R$ and $\iota_k^R$ are complex analytic.

 Recall from \cite[Proposition 4.4. (d)]{hg2011} that the product of two Silva spaces is again a Silva space.\footnote{$\Mo{\C} \times \Mo{\C}$ is the limit of $(\C^{\RT_0} (\omega_k) \times \C^{\RT_0} (\omega_k))_{k \in \N}$ where the obvious bonding maps are compact.}
  The group operations are thus defined on a closed affine subspace of a Silva spaces.
  By Proposition \ref{prop: silva} these maps will be holomorphic if and only if for each $k \in \N$ the restrictions
  $\iota_k \coloneq \iota|_{\C^{\RT_0}(\omega_k) \cap \tBGC}$ and $\cdot_k \coloneq \cdot|_{\C^{\RT_0} (\omega_k) \cap \tBGC \times\C^{\RT_0} (\omega_k) \cap \tBGC}$ are holomorphic.
 However, for each $R>1$ the maps $\iota_k$ and $\iota_k^R$ as well as the maps $\mu_k$ and $\mu_k^R$ coincide.
 As complex analyticity is a local condition, we conclude that $\cdot_k$ and $\iota_k$ are holomorphic and thus the group operations are holomorphic by Proposition \ref{prop: silva}.
 \end{proof}

 \begin{cor}\label{cor: cplx: inc}
  The inclusion $I_\C \colon \Mo{\C} \rightarrow \C^{\RT_0}$ (cf.\ Lemma \ref{lem: cI:cont}) restricts to a morphism $\alpha \colon \tBGC \rightarrow \BGC$ of complex Lie groups (where $\BGC$ is endowed with the Lie group structure of \cite[Theorem 2.1]{BS14}).
 \end{cor}

 \begin{proof}
  We already know from Lemma \ref{lem: cI:cont} that $I_\C$ is continuous and linear, whence smooth.
  By Theorem \ref{thm: tameBGp:Lie} the Lie group $\tBGC$ is a closed affine subspace of $\Mo{\C}$.
  Likewise, by \cite[Theorem 2.1]{BS14} the Lie group $\BGC$ is a closed affine subspace of $\C^{\RT_0}$. By construction of $\alpha$, the map is the restriction and corestriction of $I_\C$ to the corresponding affine subspaces.
  Hence $\alpha$ is smooth and since the group operations on $\tBGC$ arose by restricting the group operations of $\BGC$, the map becomes a morphism of Lie groups.
 \end{proof}

 \begin{rem}
  Notice that by Lemma \ref{lem:folklore} a continuous map $f \colon E \rightarrow \tBGC$ is holomorphic if and only if $\ev_\tau \circ f$ is holomorphic for each $\tau \in \RT_0$.
  Thus continuous maps to the complex tame Butcher group are holomorphic if and only if they are component wise holomorphic.
  This result is a handy criterion to establish differentiability for mappings into the complex tame Butcher group.
 \end{rem}

In Lemma \ref{prop: complexification} we have identified the complexification of $\Msp{\R}$ as $\Msp{\C}$. This allows us to identify the complex tame Butcher group as a complexification (cf.\ \cite[9.6]{HG15reg}) of the tame Butcher group.

\begin{defn}\label{defn: complexification}
 Let $G$ be a real analytic Lie group modelled on the locally convex space $E$ and $G_\C$ be a complex analytic Lie group modelled on $E_\C$.
 Then $G_\C$ is called a \emph{complexification} of $G$ if $G$ is a real submanifold of $G_\C$, the inclusion $G \rightarrow G_\C$ is a group homomorphism and for each $g \in G$, there exists an open $g$-neighbourhood $V \subseteq G_\C$ and a complex analytic diffeomorphism $\phi \colon V \rightarrow W \subseteq E_\C$ such that $\phi(V\cap G) = W\cap E$.
\end{defn}

With this notion of complexification of Lie groups, the results above entail the following corollaries (whose proofs are obvious since the complex maps restrict to real analytic mappings).

\begin{cor}\label{cor: LGP:compl}
 The complex tame Butcher group contains the tame Butcher group $\tBG$ as a real analytic Lie subgroup modelled on the Silva space $\Mo{\R}$.
 Moreover, the inclusion map is a homomorphism of groups, whence $\tBGC$ is a complexification of $\tBG$.
\end{cor}

\begin{cor}\label{cor: real: inc}
 The canonical inclusion $\alpha_\R \colon \tBG \rightarrow \BGp$ is a morphism of real analytic Lie groups.
\end{cor}

Summing up the previous results, we obtain the following commuting diagram of (complex and real) analytic Lie groups
\begin{equation}\label{eq: relation} \begin{aligned}
 \begin{xy}
  \xymatrix{
    \tBG \ar[r]^{\alpha_{\R}}    \ar[d]^\subseteq         &   \BGp  \ar[d]^\subseteq     & \text{real analytic Lie group} \ar[d]^{\text{complexification}}\\
      \tBGC \ar[r]^{\alpha}    &   \BGC   & \text{complex analytic Lie group.} \\
  }
\end{xy}
\end{aligned}
\end{equation}
 The diagram \eqref{eq: relation} suggests a close connection between the tame Butcher group and the Butcher group.
 Though the (complex) tame Butcher group is not a Lie subgroup of the (complex) Butcher group, the Lie group morphisms $\alpha$ and $\alpha_{\R}$ still yield much information.
 In fact, we have already exploited this connection to establish the Lie group structure of the tame groups: Our computation of the derivative of the group operations was carried out in the complex Butcher group. We then used an additional argument to show that this is indeed the derivative in the tame Butcher group.

 In the rest of the paper we will frequently encounter this type of reasoning.
 Before we continue investigating the structure of the tame Butcher group let us first consider certain subgroups which appear already in Butcher's seminal work \cite{Butcher72} on the Butcher group.

 \begin{defn}
  For $n \in \N$ we let $\RT^{\leq n}$ be the subset of all rooted trees $\tau \in \RT$ with fewer than $n$ nodes, i.e.\ $0< |\tau | \leq n$.
  Similarly set $\RT_0^{\leq n} \coloneq \RT^{\leq n} \cup \{\emptyset\}$.
  Now define the subgroups
  \begin{displaymath}
  \cASUB := \{ a \in \tBGC \mid a (\tau) = 0 \quad \forall \tau \in \RT^{\leq n}\}
  \end{displaymath}
  and $\ASUB := \cASUB \cap \tBG$.
 \end{defn}

 \begin{prop}\label{prop: clnorm:sbgp}
  For each $n\in \N$, $\cASUB$ is a normal closed Lie subgroup of the complex tame Butcher group and $\ASUB$ is a normal closed subgroup of the tame Butcher group.
  Further, both subgroups are split submanifolds of finite codimension.
 \end{prop}
 \begin{proof}
 An easy computation using the definition of the group operations shows that $\cASUB$ and $\ASUB$ are normal subgroups of the (complex) tame Butcher group (alternatively this is proved in \cite[Theorem 7.5]{Butcher72}).
 Note that
 \begin{displaymath}
  \cASUB = \mathrm{Ev}_n^{-1} (0) \quad \quad \text{ for } \quad \mathrm{Ev}_n \colon \tBGC \rightarrow \prod_{\tau \in \RT^{\leq n}} \C, \quad a \mapsto (a(\tau))_{\tau \in \RT^{\leq n}}.
 \end{displaymath}
 Moreover, $\mathrm{Ev}_n$ is continuous (even holomorphic) by a combination of Corollary \ref{cor: ev:cont} and Lemma \ref{lem: aff:subs}, whence the subgroups are closed.

 Recall from \cite[Remark IV.3.17]{neeb2006} that closed subgroups of infinite-dimensional Lie groups will in general not be Lie subgroups.
 However, in this case, $\cASUB$ is an affine subspace of $\Mo{\C}$. Namely, we have
 \begin{displaymath}
  \cASUB = e + \{\an \in \Mo{\C} \mid \an (\tau) = 0 \quad \forall \tau \in \RT_0^{\leq n}\}.
 \end{displaymath}
 Hence $\cASUB$ is a subgroup which is a closed submanifold, i.e.\ it is a closed Lie subgroup.
 Now the closed subspace $\{\an \in \Mo{\C} \mid \an (\tau) = 0 \quad \forall \tau \in \RT_0^{\leq n}\}$ is of finite codimension, whence it is in particular a complemented subspace of $\Mo{\C}$.
 \end{proof}

 \begin{rem}
  \begin{enumerate}
   \item Contrary to situation for finite-dimensional Lie groups, a quotient of an infinite-dimensional Lie group modulo a closed subgroup will in general not be a manifold in a canonical way.
  However, \cite[Theorem G]{hg2015} together with Proposition \ref{prop: clnorm:sbgp} asserts that $\tBGC / \cASUB$ (or $\tBG/\ASUB$) is a complex (or real) analytic finite-dimensional Lie group if $\cASUB$ (or $\ASUB$) is a regular Lie group.
  Hence as a consequence of Corollary \ref{cor: sbgp:reg} below, the group $\tBGC / \cASUB$ is a Lie group such that the canonical quotient map $q \colon \tBGC \rightarrow \tBGC/\cASUB$ is a (complex analytic) submersion.
  A similar assertion holds for $\tBG/\ASUB$.
  \item The quotient Lie groups $\tBGC / \cASUB$ and $\tBG/\ASUB$ are closely related to the (complex) Butcher group.
  It is easy to see that the group $\tBGC/\cASUB$ can be identified with the group
  \begin{displaymath}
   G^{\leq n} \coloneq \{a \in \C^{\RT_0^{\leq n}} \mid a(\emptyset)=1\}
  \end{displaymath}
  where the group operations arise by restricting the group operations of the (complex) Butcher group, i.e.~\eqref{eq: BGP:mult} and \eqref{eq: BGP:invers} to $\tau\in \RT_0^{\leq n}$.
  Notice that the (complex) Lie group $\BGC$ is the projective limit\footnote{Observe that $\BGC$ is a pro-Lie group by \cite[Section 5]{1501.05221v3}.} of the family $(G^{\leq n})_{n \in \N}$, where $G^{\leq n} \leftarrow G^{\leq n+1}$ is the restriction map.
  Similar arguments yield an analogous result for the real Lie groups $\tBG / \ASUB$.
  \end{enumerate}
 \end{rem}

 In the following sections we will discuss Lie theoretic properties of the tame Butcher groups and the subgroups $\cASUB$ and $\ASUB$.

 \section{The Lie algebra of the tame Butcher group}\label{sect: BG-LAlg}

In this section, the Lie algebra of the (complex) tame Butcher group will be determined.
Again it suffices to determine the Lie algebra of the complex tame Butcher group since the Lie algebra of the tame Butcher group arises as the real part of the complex Lie algebra.

\begin{setup}
As usual, we use the derivative of the global chart $\Psi \colon \tBGC \rightarrow \Mo{\C}, a \mapsto a-e$ to identify the Lie algebra $\Lf(\tBGC)$ of $\tBGC$ with $\Mo{\C}$ together with a suitable Lie bracket.
Recall from \cite{BS14} that $\Phi \colon \BGC \rightarrow \{a \in \C^{\RT_0} \mid a (\emptyset) = 0\} \equalscolon (\C^{\RT_0})_*$ is a global chart of the complex Butcher group.
Similar to the case of the tame Butcher group we use the derivative of $\Phi$ to identify the Lie algebra $\Lf (\BGC)$ with $(\C^{\RT_0})_*$.
\smallskip

Now the inclusion map $\alpha \colon \tBGC \rightarrow \BGC$, which is a smooth Lie group morphism by Corollary \ref{cor: cplx: inc}, such that $\Phi \circ \alpha \circ \Psi^{-1}$ is the canonical inclusion $i_\C \colon \Mo{\C} \rightarrow (\C^{\RT_0})_*$.
The identification of the Lie algebras now takes the Lie algebra morphism $\Lf (\alpha) = T_e \alpha$ to $i_\C$.
In conclusion, the Lie bracket of the complex tame Butcher group is the restriction of the Lie bracket of the complex Butcher group, i.e.\ it is given by the formula in \cite[Theorem 2.4]{BS14}.
\end{setup}

 To state the formula for the Lie bracket, recall some notation.

 \begin{defn}\label{defn: split}
  Let $\tau \in \RT_0$ be a rooted tree.
  We define the \emph{set of all splittings} as
    \begin{displaymath}
     \SP{\tau} \coloneq \{s \in \OST (\tau) \mid \tau \setminus s \text{ consists of only one element}\}
    \end{displaymath}
  Furthermore, define the set of \emph{non-trivial splittings}
  $\SP{\tau}_1 \coloneq \{\theta \in \SP{\tau} \mid \theta \not \in \{\emptyset, \tau\}\}$.
  Observe that for each tree $\tau$ the order of trees in $\SP{\tau}_1$ is strictly less than $|\tau|$.
 \end{defn}

\begin{thm}\label{thm: tameBG:LA}
  The Lie algebra of the complex tame Butcher group is $(\Msp{\C}, \LB )$, where the Lie bracket $\LB[\an , \bn]$ for $\an ,\bn \in \Msp{\C}$ is
  \begin{equation}\label{eq: LB}
   \LB[\an,\bn] (\tau) = \sum_{s \in \SP{\tau}_1} \left(\bn(s_{\tau}) \an (\tau \setminus s) - \an(s_{\tau}) \bn (\tau \setminus s)  \right)
  \end{equation}
 By \eqref{eq: LB} $\LB$ restricts to a Lie bracket on $\tBG$.
 The Lie algebra of the tame Butcher group is $(\Msp{\R}, \LB )$, with the bracket induced by $\Lf(\tBGC)$ on the subspace $\Msp{\R}$.
\end{thm}

\begin{rem}
 The crucial point of the whole construction is that the Lie bracket given via \eqref{eq: LB} is continuous with respect to the Silva space topology.
\end{rem}

\begin{cor}
 For $n\in \N$ the Lie algebra of the Lie subgroup $\cASUB \subseteq \tBGC$ (cf.\ Proposition \ref{prop: clnorm:sbgp}) is the Lie subalgebra
 \begin{displaymath}
   \Lf (\cASUB) = \{\an \in \Lf (\tBGC) \mid \an (\tau) = 0 \quad \forall \tau \in \RT^{\leq n}_0 \}.
 \end{displaymath}
 Further, the Lie algebra of $\ASUB$ is given by $\Lf (\ASUB) = \Lf (\cASUB) \cap \Lf (\tBG)$.
\end{cor}

\begin{proof}
 Follows directly from the description of the subgroups as affine subspaces in the proof of Proposition \ref{prop: clnorm:sbgp}.
\end{proof}

\section{Regularity properties of the tame Butcher group}\label{sect: regularity}

 Let us now discuss regularity properties of the (complex) tame Butcher group.
 Since we also want to establish regularity properties for the complex Lie group $\tBGC$ several comments are needed:

 Recall that holomorphic maps are smooth with respect to the underlying real
 structure by \cite[Proposition 2.4]{hg2002a},
 whence $\tBGC$ carries the structure of a real Lie group.
 Now the complex Lie group $\tBGC$ is called regular, if the underlying real Lie group is regular (in the sense outlined in the introduction).
 Thus for this section we fix the following convention:
 \medskip

 \textit{Unless stated explicitly otherwise, all complex vector spaces in this section are to be understood as the underlying real locally convex vector spaces.
  Moreover, differentiability of maps is understood to be differentiability with respect to the field $\R$.}
 \medskip

 In \cite{BS14} we proved that the (complex) Butcher group $\BGC$ is $C^k$-regular for each $k\in \N_0$.
 Fix $k \in \N_0 \cup \{ \infty\}$ and let $\Evol_{\BGC} \colon C^k ([0,1],\Lf (\BGC)) \rightarrow C^{k+1} ([0,1],\BGC)$ be the curve evolution map,
 i.e.\ the curve $\Evol_{\BGC} (\eta)$ solves the differential equation of regularity for $\eta$.
 Let us use curves as candidates for the evolution of smooth curves in $C^k ([0,1] , \Lf (\tBGC))$.

 \begin{setup}[Semiregularity of $\tBGC$ via $\BGC$]\label{setup: comp:diffeq}
 By Corollary \ref{cor: cplx: inc} the continuous linear inclusion $I_\C \colon \Mo{\C} \rightarrow \K^{\RT_0}$ restricts to a Lie group morphism $\alpha \colon \tBGC \rightarrow \BGC$.
 Its derivative $\Lf (\alpha) \colon \Lf (\tBGC) \rightarrow \Lf (\BGC)$ is the inclusion of $\Lf (\tBGC)$ into $\Lf (\BGC)$ (albeit the topology $\Lf (\alpha)$ induced on its image is finer than the subspace topology).
 Each $C^k$-curve $\eta \colon [0,1] \rightarrow \Lf (\tBGC)$ yields a $C^k$-curve $\Lf (\alpha) \circ \eta \colon [0,1] \rightarrow \Lf (\BGC)$.
 As a consequence of \cite[Lemma 10.1]{HG15reg} we see that if $\eta \in C^k ([0,1] , \Lf (\tBGC))$ admits a $C^{k+1}$-evolution $\gamma_\eta \colon [0,1] \rightarrow \tBGC$ then it must satisfy
  \begin{equation}\label{eq: evol:eq}
   \alpha \circ \gamma_\eta = \Evol_{\BGC} (\Lf (\alpha ) \circ \eta)
  \end{equation}
 From this formula and \cite[Section 4]{BS14} we obtain a description of $\gamma_\eta$ in terms of its components.
 A combination of \eqref{eq: evol:eq} and \cite[Eq. (8)]{BS14} yields for $\tau \in \RT_0$ the formula
  \begin{align}\label{eq: diffeq:comp}
   \gamma_\eta' (t) (\tau) = \ev_{\tau} \circ (\alpha \circ \gamma_\eta)' (t) &= \ev_{\tau} (\eta (t)) +  \sum_{s \in \SP{\tau}_1} \ev_{s_{\tau}}(\gamma_\eta (t)) \eta (t)(\tau \setminus s).
  \end{align}
 For each $\tau \in \RT_0$ the differential equation \eqref{eq: diffeq:comp} is a differential equation in $\C$ which can be solved by integrating the right-hand side if  the $\gamma_\eta (\cdot) (s_\tau)$ are known for all $s \in \SP{\tau}_1$.
 As the trees in $\SP{\tau}_1$ have strictly less nodes than $\tau$, these differential equations can be solved iteratively.
 Summing up, the group $\tBGC$ will be $C^k$-semiregular if we can show that the solutions to the system of differential equations \eqref{eq: diffeq:comp} are exponentially bounded for all $\eta \in \tBGC$.
 \end{setup}

 \begin{lem}\label{lem: factor:evolution}
  Let $\eta \in C^0([0,1] , \Lf (\tBGC))$ be a curve with $\sup_{t \in [0,1]}\norm{\eta (t)}_{\omega_m} \leq 1$ for some $m \in \N$.
  Consider the map $\gamma_\eta (t) \coloneq \Evol_{\BGC} (\Lf (\alpha) \circ \eta) (t) \colon [0,1] \rightarrow \BGC$.
  Then
$\gamma_\eta$ satisfies $\sup_{t \in [0,1]} \norm{\gamma_\eta (t)-e}_{\omega_{m+1}} \leq 1$.
  In particular $\gamma_\eta$ takes its image in $\C^{\RT_0} (\omega_{m+1})$ and factors through $\tBGC$.
 \end{lem}

 \begin{proof}
Use the component wise formula \eqref{eq: diffeq:comp} for the derivative of $\gamma_\eta$ to obtain an estimate of the growth of $\gamma_\eta$.
  We claim that $\gamma_{\eta}$ satisfies the inequality
  \begin{equation}
  |\gamma_{\eta}(t)(\tau)|\omega_m(\tau) \leq P_{|\tau|}(t) , \quad \forall t \in [0,1]
  \label{eq: growthest}
  \end{equation}
  where $P_{k}(t) = \frac{1}{k}((1+t)^k-1)$.

  Note that $P_{k+1}(t)\geq P_k(t)$ for all $k\in \N$ and $t\ge 0$, and that
  for $t\in [0,1]$ equation \eqref{eq: growthest} implies $|\gamma_{\eta}(t)(\tau)|\omega_m(\tau) \leq 2^{|\tau|}= \frac{\omega_m(\tau)}{\omega_{m+1}(\tau)}$.
  Hence if the claim holds, this implies $\sup_{t \in [0,1]} \norm{\gamma_\eta (t)}_{\omega_{m+1}} \leq 1$ and the assertion follows.

  \textbf{Proof of the claim}: For the one-node tree $\bullet$, $\omega_m(\bullet)=\tfrac{1}{2^m}$ holds and thus
    \begin{align*}
     \abs{\gamma_\eta (t) (\bullet)}\omega_{m} (\bullet) &= \abs{\int_0^t \eta (s)(\bullet)\di s} \omega_m (\bullet) \leq t \sup_{s \in [0,1]} \abs{\eta (s)(\bullet)} \omega_m (\bullet) \\
						       &\leq t \sup_{s\in [0,1]} \norm{\eta (s)}_{\omega_m} \le t = P_{|\bullet|} (t).
    \end{align*}
  Now, let $\tau$ be a tree with $|\tau|>1$ and assume that \eqref{eq: growthest} holds for all trees $\tau'$ with $|\tau'|<|\tau|$.
  By integrating $\eqref{eq: evol:eq}$ and taking absolute values, we get
  \begin{align*}
     \abs{\gamma_\eta (t) (\tau)}=&\abs{\int_0^t\ev_{\tau} (\eta (r)) +  \sum_{s \in \SP{\tau}_1} \ev_{s_{\tau}}(\gamma_\eta (r))\eta (r)(\tau \setminus s)\di r} \\
      \le &\int_0^t\abs{\ev_{\tau} (\eta (r))} +  \sum_{s \in \SP{\tau}_1}  \abs{\ev_{s_{\tau}}(\gamma_\eta (r))} \cdot \abs{\eta (r)(\tau \setminus s)}\di r.
  \end{align*}
  Multiplying both sides with $\omega_m(\tau)$ and using that $\omega_m$ is multiplicative, we get
  \begin{align*}
   \abs{\gamma_\eta (t) (\tau)}\omega_m(\tau)
        \le &\int_0^t\abs{\ev_{\tau} (\eta (r))}\omega_m(\tau) \\
         &+  \sum_{s \in \SP{\tau}_1}\abs{\ev_{s_{\tau}}(\gamma_\eta (r))}\omega_m(s_\tau) \cdot \abs{\eta (r)(\tau \setminus s)}\omega_m(\tau \setminus s)\di r \\
         \le &\int_0^t 1 + \sum_{s\in \SP{\tau}_1} P_{|s_\tau|}(r)\cdot 1 \di r .
  \end{align*}
  We now use that for all $s\in \SP{\tau}_1$ the tree $s_\tau$ satisfies $|s_\tau| \le |\tau|-1$ and that $\abs{\SP{\tau}_1} = |\tau|-1$ to get
  \begin{align*}
   \abs{\gamma_\eta (t) (\tau)}\omega_m(\tau) \le& \int_{0}^t 1 + (|\tau|-1)P_{|\tau|-1}(r) \di r \\
                                              \le& \int_0^t (1+t)^{|\tau|-1} \di r = \frac{1}{|\tau|} ((1+t)^{|\tau|}-1) = P_{|\tau|}(t).\quad \qedhere
  \end{align*}
\end{proof}

\begin{thm}
Let $G$ be either $\tBGC$ or $\tBG$, then $G$ is $C^1$-regular and $C^0$-semiregular.
Moreover, $\evol_{\tBGC}$ is complex analytic and $\evol_{\tBG}$ is real analytic.
\end{thm}

\begin{proof}
We begin with the complex case.
Consider a curve $\eta \in C^0([0,1] , \Lf (\tBGC))$ with $\sup_{t \in [0,1]}\norm{\eta (t)}_{\omega_m} \leq 1$ for some $m \in \N$.
Define $L_m \coloneq \{a \in \C^{\RT_0} (\omega_m)\mid a (\emptyset) = 0\}$.
By Lemma \ref{lem: factor:evolution} the map $\gamma_\eta = \Evol_{\BGC} (\Lf(\alpha) \circ \eta) - e$ takes its values in a ball in $L_{m+1}$.
Denote by $K$ the closure of this ball in $L_{m+2}$. Since the bonding map $L_{m+1} \rightarrow L_{m+2}$ is a compact operator $K$ is compact.

Since $K$ is compact the subspace topology on $K$ coincides with the topology induced by the continuous inclusion $L_{m+2} \rightarrow (\C^{\RT_0})_* = \{a \in \C^{\RT_0} \mid a(\emptyset)=0\}$.
As $\gamma_\eta - e$ is continuous as a map to $(\C^{\RT_0})_*$, it is thus continuous as a map to $L_{m+2}$.
We conclude that $\gamma_\eta -e$ is continuous as a map to $\Mo{\C}$.
Now \cite[Lemma 7.10]{HG15reg} implies that indeed $\gamma_\eta = \Evol_{\tBGC} (\eta)$.

Using that the topologies on the compact set $K+e$ induced by $\BGC$ and $\tBGC$ coincide, we deduce from the continuity of $\evol_{\BGC}$ that also $\evol_{\tBGC} (\eta) \in K+e \subseteq \tBGC$ is continuous as a function of $\eta$, where $\eta$ varies in $P_m \coloneq C ([0,1], B^{L_m}_1 (0)))$.

Hence for every step of the inductive limit $\C^{\RT_0} (\omega_m)$ we obtain a zero-neighbourhood $P_m \subseteq C([0,1] , \C^{\RT_0} (\omega_m))$ such that for every $\eta$ in $P_m$ an evolution $\Evol (\eta) (=\gamma_\eta)$ exists and the evolution operator $\evol_m \colon P_m \rightarrow \tBGC$ is continuous.
Furthermore, we have a Lie group morphism $\alpha \colon \tBGC \rightarrow \BGC$ to a $C^0$-regular Lie group (cf.\ \cite[Theorem C]{BS14}) which separates the points.
We can thus apply \cite[Theorem 15.5]{HG15reg} to see that $\tBGC$ is $C^1$-regular and $C^0$-semiregular.
As the Lie group is complex analytic, \cite[Lemma 9.8]{HG15reg} shows that $\evol_{\tBGC}$ is even complex analytic.

Now for the real case, observe that $\tBGC$ is a complexification of $\tBG$ by Corollary \ref{cor: LGP:compl}.
Hence the assertions for $\tBG$ follow from \cite[Corollary 9.10]{HG15reg}.
\end{proof}

 Note that the evolution in the complex tame Butcher group is, up to identification, given by the evolution in the complex Butcher group.
 Hence, \cite[Eq. (8)]{BS14} yields a formula for the evaluation of each component of the evolution of a curve $\eta \in C([0,1], \Lf (\tBGC))$ (cf.\ also \eqref{eq: diffeq:comp}).
 Namely, for each $\tau \in \RT_0$ we obtain
 \begin{equation}\label{eq: tameevolution}
  \ev_\tau \circ \Evol_{\tBGC} (\eta) (t) = \int_0^t \eta (r)(\tau) + \sum_{s \in \SP{\tau}_1} \ev_{s_\tau} (\Evol_{\tBGC} (\eta)(r)) \eta (r)(\tau\setminus s) \di r.
 \end{equation}
 Armed with this knowledge, we can assert that the normal subgroups of the tame Butcher group discussed in Proposition \ref{prop: clnorm:sbgp} are regular Lie subgroups.

\begin{cor}\label{cor: sbgp:reg}
 For each $n \in \N$, the Lie groups $\cASUB$ and $\ASUB$ of $\tBGC$ are $C^0$-semiregular and $C^1$-regular.
\end{cor}

\begin{proof}
 For $\eta \in C([0,1],\Lf (\cASUB)$ we consider the formula \eqref{eq: tameevolution} for $\Evol_{\tBGC} (\eta) (t)$.
 Since non-trivial splittings of a rooted tree $\tau$ have strictly less nodes than $\tau$, we see that all terms in \eqref{eq: tameevolution} vanish if $\eta (t)(\tau)$ and $\eta(t) (\tau\setminus s)$ vanishes.
 Thus $\ev_\tau \circ \Evol_{\tBGC} (\eta) (t) =0$ for all $\tau \in \RT$ with $|\tau |\leq n$ if $\eta \in C([0,1], \Lf(\cASUB))$.
 We deduce that $\cASUB$ is $C^0$-semiregular. As $\cASUB$ is a closed subgroup of the $C^1$-regular Lie group $\tBGC$, \cite[Lemma B.5]{SchmedingWockel15} implies that $\cASUB$ is $C^1$-regular.
 Analogously one establishes the regularity conditions for $\ASUB$.
\end{proof}

\section{The tame Butcher group is not (locally) exponential}
\label{sect: nonexp}

In this section, we show that the tame Butcher group is neither exponential nor even locally exponential.
This is in contrast to the full Butcher group, which is exponential (cf.\ \cite[Section 5]{BS14}).
From the point of view of numerical analysis, the lack of exponentiality of the tame Butcher group is reflected by the fact that so-called modified equations are in general formal series that do not converge.
The construction of an element which is not exponentially bounded but whose image with respect to the exponential map is contained in $\tBG$ is, in essence, equivalent to the argument given in \cite[Example 7.1]{HLW2006}.
However, the approach as presented in the present paper stays entirely in the realm of tree coefficients.
Before we construct elements near the identity which are not contained in the image of the exponential map, let us prepare the proof with several remarks.

\begin{rem}
 To prove that the (complex) tame Butcher group is not (locally) exponential it obviously suffices to prove this for the tame Butcher group $\tBG$.
 Hence we will construct elements in $\tBG$ which are arbitrarily near the identity but not contained in the image of $\exp_{\tBG}$.
\end{rem}

\begin{defn}[{Lie group exponential \cite[Definition II.5.1]{neeb2006}}]
The \emph{Lie group exponential} $\exp_{\tBG} \colon \Lf (\tBG) \rightarrow \tBG$ is the unique map such that for each $x \in \Lf (\tBG)$ the assignment $\gamma_x (t) \coloneq \exp_{\tBG} (tx)$ is a one parameter group of $\tBG$ with $\gamma_x' (0) = x$.
\end{defn}

 It is easy to see that $\gamma_x (t) = \exp_{\tBG} (tx)$ solves the initial value problem
  \begin{equation}\label{eq: Lie:expo}
   \begin{cases}
    \gamma_x'(t) &= \gamma_x (t).x= T_e\lambda_{\gamma_x (t)} (x),  \\
    \gamma_x (0) &= e,
   \end{cases}
  \end{equation}
 where $\lambda_{\gamma_x (t)}$ is the left-translation by $\gamma_x (t)$.
 Since $\tBG$ is a regular Lie group, the differential equation \eqref{eq: Lie:expo} admits a unique solution.\footnote{See \cite{milnor1983} or alternatively \cite[Remark II.5.3]{neeb2006} combined with \cite[Proposition 1.3.6]{dahmen2011} (also cf.\ \cite[Lemma 7.2]{HG15reg}).}

 We will again exploit that the tame Butcher group is closely connected to the Butcher group.

 \begin{rem}\label{rem: candidates}
  In \cite{BS14} we have shown that the Butcher group $\BGp$ is an exponential Lie group, i.e.\ its Lie group exponential is a diffeomorphism $$\exp_{\BGp} \colon \Lf (\BGp) \rightarrow \BGp.$$
  Recall from Corollary \ref{cor: real: inc} that the inclusion $\alpha_\R \colon \tBG \rightarrow \BGp$ is a Lie group morphism.
  By construction the induced map $\Lf (\alpha_\R) \colon \Lf (\tBG) \rightarrow \Lf(\BGp)$ coincides with the inclusion of $\Lf (\tBG)$ into $\Lf(\BGp)$.
  Now the naturality of the Lie group exponential maps asserts that for all $x \in \Lf (\tBG)$ we have with the obvious identifications
   \begin{displaymath}
   \exp_{\tBG} (x)= \alpha_\R \circ \exp_{\tBG} (x) = \exp_{\BGp} \circ \Lf (\alpha_\R)(x) = \exp_{\BGp} (x).
   \end{displaymath}
 Thus for an element $a \in \tBG$ the only candidate for a logarithm $\log_{\tBG} (a)$ (i.e.\ an element in $\Lf (\tBG)$ with $a = \exp_{\tBG} (\log_{\tBG} (a))$) is $\log_{\BGp} (\alpha_\R (a)) = \log_{\BGp} (a)$.
 \end{rem}

 Our aim is now to construct a logarithm for elements in $\tBG$ whose coefficients on certain trees are not exponentially bounded.
 The idea is that for the so called ``bushy trees'' we can derive an explicit formula for the coefficients of a certain logarithm.

 \begin{nota}
  Let $\tau_n$ be the $n$th bushy tree, i.e.~the tree obtained by connecting $n-1$ singular nodes to a common root.
  To illustrate the special form of these trees, we depict the first bushy trees below:
  \begin{displaymath}
     \begin{tikzpicture}[dtree,scale=1.5]
             \node[dtree black node] {}
                        ;
           \end{tikzpicture}, \quad \quad
               \begin{tikzpicture}[dtree,scale=1.5]
             \node[dtree black node] {}
             child { node[dtree black node] {} }
             ;
           \end{tikzpicture}, \quad \quad
           \begin{tikzpicture}[dtree, scale=1.5]
                              \node[dtree black node] {}
                              child { node[dtree black node] {}
                              }
                              child {node[dtree black node] {}}
                              ;
                            \end{tikzpicture}, \quad \quad
           \begin{tikzpicture}[dtree, scale=1.5]
                     \node[dtree black node] {}
                     child { node[dtree black node] {}
                     }
                     child {node[dtree black node] {}}
                     child { node[dtree black node] {}
                     }
                     ;
                   \end{tikzpicture}, \quad \quad
          \begin{tikzpicture}[dtree, scale=1.5]
                     \node[dtree black node] {}
                     child { node[dtree black node] {}
                     }
                     child {node[dtree black node] {}}
                     child { node[dtree black node] {}}
                      child { node[dtree black node] {}}
                     ;
                   \end{tikzpicture},\quad \ldots
  \end{displaymath}
 \end{nota}

\begin{prop}
\label{prop: nonexpcoeff}
For $\ve \in \R \setminus \{0\}$ let $a_\ve \in \tBG$ be defined by
\begin{displaymath}
   a_{\ve}(\tau) = \begin{cases}
   1, & \text{if $\tau= \emptyset$,} \\
   \ve,   & \text{if $\tau=\onenode$,} \\
   0 & \text{else,}
   \end{cases}
   \end{displaymath}
and $\eta_\ve = \log_{\BGp}(a_\ve)\in \Lf(\BGp)$ be the logarithm of $a_\ve$ in the full Butcher group.
Then $\eta_\ve(\tau_n)=\ve^n B_{n-1}$, where $B_{n}$ is the $n$th Bernoulli number.
\end{prop}
\begin{proof}
Fix $\ve \in \R \setminus \{0\}$ and let $\eta_\ve = \log_{\BGp}(a_\ve)$.
To shorten the notation we denote by $\gamma \colon [0,1] \rightarrow \BGp$ the smooth curve which solves \eqref{eq: Lie:expo} (in the Lie group $\BGp$).
By \cite[Section 4]{BS14} and the above remarks,
$a_\ve = \gamma(1),$
and $\gamma$ solves the differential equation
\begin{equation}
\label{eq: nonexp:leftevol}
\begin{cases}
 \gamma' (t) (\tau) &= \displaystyle \sum_{s \in \OST(\tau)_1}\eta_\ve(s_{\tau}) \prod_{\theta \in s \setminus \tau}\gamma(t)(\theta)  \\
 \gamma(0)(\tau) &= e(\tau)
 \end{cases}
\end{equation}
where we denote by $\OST(\tau)_1$ the non-empty ordered subtrees of $\tau$.

If $\tau = \onenode = \tau_1$, equation \eqref{eq: nonexp:leftevol} simplifies to
\begin{displaymath}
\gamma' (t) (\onenode) = \eta_\ve(\onenode).
\end{displaymath}
Combining this with $\gamma(1)(\onenode)= a_{\ve}(\onenode)=\ve$, we get
\begin{equation}
\begin{gathered}\label{eq: nonexp:onenode}
\eta_\ve(\onenode)=\ve, \qquad \gamma(t)(\ve)=\ve t.
\end{gathered}
\end{equation}

In the rest of the proof, we set  $\tau =\tau_n$ in \eqref{eq: nonexp:leftevol} for $n>1$. the non-empty ordered subtrees $s_{\tau_n}$ are of the form $\tau_{k}$ for $1\le k\le n$, and $\tau_n \setminus s$ is the forest consisting of $n-k$ one-node trees.
Using this observation, we collect recurring terms in \eqref{eq: nonexp:leftevol} to get
\begin{equation}
\label{eq: nonexp:ode}
\begin{aligned}
 \gamma' (t) (\tau_n) &= \sum_{k=1}^{n} \binom{n-1}{k-1} \eta_\ve(\tau_k)\left(\gamma(t)(\onenode)\right)^{n-k} \\
                      &=  \sum_{k=1}^{n} \binom{n-1}{k-1} \eta_\ve(\tau_k)(\ve t)^{n-k}.
                      \end{aligned}
\end{equation}
Integrating over $t$ and using $a_\ve=\gamma(1)$, we get the following equality
\begin{equation}
\label{eq: nonexp:eqsys}
\begin{aligned}
0 = a_{\ve}(\tau_n) = \gamma(1)(\tau_n) &= \sum_{k=1}^{n} \frac{1}{n-k+1}\binom{n-1}{k-1} \eta_\ve(\tau_k)\ve^{n-k}, \\
&= \frac{\ve^n}{n} \sum_{k=1}^{n} \binom{n}{k-1} \eta_\ve(\tau_k)\ve^{ -k},
\end{aligned}
\end{equation}
for all $n\ge 2$.
In the last equality of \eqref{eq: nonexp:eqsys}, we have used $\frac{1}{n-l}\binom{n-1}{l}=\frac{1}{n}\binom{n}{l}$ and pulled $\ve^n/n$ out of the sum.
Now the linear equations \eqref{eq: nonexp:eqsys} (for $n \in \N$) can be recursively solved for $\eta_\ve(\tau_k)$, $k\ge 2$, when $\eta_\ve(\tau_1)$ is known.

It is well known (see \cite[Chapter 15 Lemma 1]{MR1070716}) that the Bernoulli numbers $B_k, k \in \N$ satisfy
\begin{displaymath}
\sum_{k=0}^{n-1} \binom{n}{k} B_k =0.
\end{displaymath}
Therefore, the general solution to \eqref{eq: nonexp:eqsys} is $\eta_{\ve}(\tau_n)=C B_{n-1}\ve^n,$ for $n=1,2,\dotsc$ From \eqref{eq: nonexp:onenode}, the constant $C$ is seen to be 1.
\end{proof}

\begin{prop}
 Neither the complex tame Butcher group $\tBGC$, nor the tame Butcher group $\tBG$ is (locally) exponential.
\end{prop}

\begin{proof}
 We have to prove that the Lie group exponential map is not a (local) diffeomorphism near the identity element.
 To see this, denote by $a_\ve$ the elements of $\tBG$ introduced in Proposition \ref{prop: nonexpcoeff} and consider the curve
  \begin{displaymath}
   a \colon \R \rightarrow \tBG , \quad \ve \mapsto \begin{cases}
                                                   e &\text{if } \ve =0,\\
                                                   a_\ve &\text{else}.
                                                  \end{cases}
  \end{displaymath}
 Clearly $a$ is a continuous curve through the identity element $e$ of $\tBG$.
 We will now prove that the only point on the curve which is contained in the image of $\exp_{\tBG}$ is the identity.
 From this it follows that the image of $\exp_{\tBG}$ does not contain an identity neighbourhood, whence the (complex) Butcher group is not (locally) exponential.

 Proposition \ref{prop: nonexpcoeff} asserts that for $\ve \neq 0$ the coefficients of $a(\ve) = a_\ve$ on the bushy tree $\tau_n$ are given by $\ve^nB_{n-1}$.
 Now \cite[Chapter 15 Eq. (7)]{MR1070716} shows that $|B_{2n}| > \frac{2(2n)!}{(2\pi)^{2n}}$, i.e. $B_{2n}$ grows factorially.
 It then follows from the proposition that for all $\ve\ne 0$, $\log_{\BGp}(a_{\ve})$ is not exponentially bounded, whence it is not contained in $\Lf (\tBG)$.
 Following Remark \ref{rem: candidates} we deduce that there $a (\ve)$ is not contained in the image of $\exp_{\tBG}$.
\end{proof}

\section{Lie group morphisms to germs of diffeomorphisms}

In this section we construct Lie group morphisms from the tame Butcher group into Lie groups of germs of diffeomorphisms.
The idea is that these Lie group morphisms are directly related to the construction of numerical approximations to solutions of ordinary differential equations.
Namely, we will realise elements in the tame Butcher group as their associated  B-series and see that these yield diffeomorphisms near equilibrium points of ordinary differential equations.
The target Lie groups envisaged here are Lie groups of germs of diffeomorphisms on Banach spaces which arise as special cases of \cite[Theorem B]{MR2670675}.

Before we begin note that by definition, a B-series is constructed with respect to a pre-chosen vector field (cf.\ Definition \ref{defn: Bseries}).
Hence the Lie group morphisms constructed in the following depend on the choice of a certain vector field.
It will turn out that the construction is closely related to a Lie group morphism which sends elements in the Butcher group to formal diffeomorphisms.
Thus we will first describe the construction of the morphism from the (full) Butcher group to formal diffeomorphisms.
To this end recall the following.

\begin{setup}{{\cite[Example IV.1.14]{neeb2006}}}\label{setup: formdiff}
 Let $\K \in \{\R,\C\}$ and $F$ be a Banach space over $\K$.
 We write $\Gf (F)$ for the \emph{group of formal diffeomorphisms} of $F$ fixing $0$. Elements in the group are represented by formal power series of the form
  \begin{displaymath}
   \varphi(x) = g.x + \sum_{k \geq 2 } p_k(x)
  \end{displaymath}
 with $g \in \GL (F)$ (an invertible continuous linear Operator on $F$) and $p_k$ a homogeneous polynomial of degree $k$ on $F$.\footnote{Recall from \cite{BS71a} that a mapping $p \colon F \rightarrow F$ is called \emph{homogeneous polynomial of degree $k$} if there is a symmetric $k$-linear map $\overline{p} \colon F^k \rightarrow F$ with $p(x) = \overline{p} (x,x,\ldots,x)$.} Here the group operation is given by composition of power series.
 We call $\varphi$ \emph{pro-unipotent} if $g=\id_F$ and note that the subgroup $\Gf_1 (F)$ of all pro-unipotent formal diffeomorphism form a pro-nilpotent Lie group (which even admits a global manifold chart, cf.\ \cite[Example IV.1.13]{neeb2006})
 \begin{displaymath}
   \Gf_1 (F) = \lim_{\leftarrow} G_k \quad \text{ with } G_k \coloneq \left\{x+ \sum_{n=2}^k p_n (x)\middle| \begin{aligned}\text{for } 2 \leq i \leq n, p_i \text{ is a} \\ \text{ polynomial of degree } i \end{aligned}\right\}  .
 \end{displaymath}
 Here $G_k$ is the nilpotent Banach Lie group obtained by defining multiplication to be composition modulo terms of order greater than $k$.
 Then the group $\Gf (F)$ of all formal diffeomorphisms of $E$ fixing $0$ is the semidirect product
  \begin{displaymath}
   \Gf (F)  = \Gf_1 (F) \rtimes \GL (F)
  \end{displaymath}
 where the Banach Lie group $\GL (F)$ acts by conjugation. As this action is smooth $\Gf (F)$ is a \Frechet Lie group.
 In the special case $F = \K^n$ for some $n \in \N$ elements in the group are represented by the familiar power series expansions
 \begin{displaymath}
  \varphi (x) = g.x + \sum_{|\mathbf{m}| >1} c_{\mathbf{m}} x^{\mathbf{m}}
 \end{displaymath}
 with $\mathbf{m} = (m_1 , \ldots ,m_n) \in \N_0^m$, $|\mathbf{m}| \coloneq m_1 + m_2 + \ldots + m_n$ and $x^{\mathbf{m}} \coloneq x_1^{m_1} x_2^{m_2} \cdots x_n^{m_n}$.\footnote{Actually \cite{neeb2006} treats only the case $F = \K^n$. However, it is easy to see that all arguments used to construct the Lie group structure generalise verbatim to the case of an arbitrary Banach space $F$.}
\end{setup}

We will now construct a Lie group morphism from the Butcher group (and thus also from the tame Butcher group) into the group of formal diffeomorphisms on a Banach space.

 Let us fix now some additional data needed to construct B-series which correspond to elements in a Lie group of germs of diffeomorphisms.

 \begin{setup}\label{setup: global}
  Choose and fix once and for all the following:
  \begin{itemize}
   \item a complex Banach space $(E, \norm{\cdot})$ together with $y_0 \in E$,
   \item a holomorphic map $f \colon E \supseteq U \rightarrow E$, defined on an open $y_0$-neighbourhood $U$.
  \end{itemize}
 As in Definition \ref{defn: Bseries} we define elementary differentials and B-series depending on $f$. However, since we fixed $f$, we will suppress $f$ in the notation and write:
  \begin{displaymath}
   B (a,y,h) = y + \sum_{\tau \in \RT} \frac{h^{|\tau|}a(\tau)}{\sigma(\tau)} F(\tau)(y) \quad \text{for } (a,y,h) \in \BGC \times U \times \C .
  \end{displaymath}
 Furthermore, recall from \cite[III. Theorem 1.10]{HLW2006} that the product in $\tBGC$ corresponds to composition of B-series, i.e.\ on the level of formal series we have
 \begin{equation}\label{eq: Bseries:insert}
  B (b , B(a,h,y),h) = B (a\cdot b,y,h) \text{ for all } a,b \in \tBGC .
 \end{equation}
  \end{setup}

\begin{rem}[Warning]
The multiplication of two elements in the (tame) Butcher group corresponds (by convention) to composition of power series in opposite order (see \eqref{eq: Bseries:insert}).
This is somewhat annoying if one wants to construct Lie group morphisms into groups of diffeomorphisms whose product is given by composition of maps.
Due to the different order of composition we can thus only construct Lie group antimorphisms, i.e.\ Lie group morphisms from the (tame) Butcher group in the opposite group.
However, this problem is unavoidable since we prefer to stick to the conventions on the group products established in the literature.
\end{rem}

\begin{rem}
 \begin{enumerate}
 \item In our B-series, we allow the time step $h$ to be a complex parameter.
 This usually has no direct interpretation in numerical analysis (where $h$ is chosen in $\R \setminus \{0\}$), but is required for technical reasons.
 We will later be able to restrict to $h \in \R$.
 \item Assuming that the map $f$ from \ref{setup: global} is holomorphic enables us to employ the powerful estimate for B-series from Proposition \ref{prop: convergence}.
 This estimate will be an indispensable tool in the following constructions.
 \end{enumerate}
\end{rem}

\begin{prop}\label{prop: formalmor}
 Let $E$ be a complex Banach space and $y_0 \in E$ fixed.
 Consider a holomorphic map $f \colon E \supseteq U \rightarrow E$ on a $y_0$-neighbourhood $U$.
 Then
 \begin{enumerate}
  \item The assignment
 \begin{displaymath}
   \formalmor_f \colon \BGC \rightarrow \Gf (E \times \C),\quad a \mapsto \left[(x,h) \mapsto (B (a,x+y_0,h)-y_0,h)) \right]
  \end{displaymath}
 is a complex analytic antimorphism of Lie groups (i.e.\ it is a Lie group morphism to the opposite group)
 \end{enumerate}
 Let $g \colon F \supseteq W \rightarrow F$ be a real analytic map such that $E$ is a complexification of $F$, $y_0 \in W$ and $f \colon E \supseteq U \rightarrow E$ its holomorphic extension, i.e.\ $f|_{U \cap F} = g$.
 \begin{enumerate}
  \item[{\upshape (b)}] The morphism $\formalmor_f$ restricts to a real analytic antimorphism of Lie groups
  \begin{displaymath}
  \formalmor_{g} \colon \BGp \rightarrow \Gf (F \times \R) ,\quad a \mapsto \formalmor_f (a)|_{F \times \R}.
 \end{displaymath}
 \end{enumerate}
\end{prop}

\begin{proof} \begin{enumerate}
               \item Since $f$ is  holomorphic the elementary differential $F(\tau)$ is complex analytic in $y$ for each $\tau \in \RT$.
 Develop each $F(\tau)$ into its Taylor series around $y_0$, i.e.\ $F(\tau)(y) = \sum_{k=1}^\infty p^\tau_k (y-y_0)$ where $ p^\tau_k$ is a homogeneous polynomial of degree $k$.
 Further, for the one node tree $\bullet$ we have $F(\bullet)(y) = f(y)$ and thus the term of degree $0$ in the Taylor expansion of $F(\bullet)$ is $p^\bullet_0 (y) = f(y_0)$.

 We conclude that $B (a,y,h)-y_0$ is a formal power series in the variables $x\coloneq y-y_0$ and $h$ which maps $(x,h)=(0,0)$ to $(0,0)$ in $E\times \C$.
 Now $\sigma(\bullet)=1$, whence the first order term of $B (a,y,h)-y_0$ is
  \begin{displaymath}
   \left(x + h a(\bullet) f(y_0) ,h\right) = \begin{pmatrix}
                                                                 \id_E & a(\bullet) f(y_0) \\0 &1
                                                                \end{pmatrix}. (x,h).
  \end{displaymath}
 In particular, we see that $(\formalmor_f (a,x+y_0,h)-y_0,h)$ is contained in $\Gf (E\times \C)$.
 Now the product in $\Gf(E\times \C)$ is composition of formal power series, i.e.\ the first component of $\formalmor_f (a) \cdot \formalmor_f (b)$ is just $B(a,B(b,x+y_0,h),h)-y_0$.
 Thus \eqref{eq: Bseries:insert} implies that $\formalmor_f (a) \cdot \formalmor_f (b) = \formalmor_f (b\cdot a)$ holds and $\formalmor_f$ is a group anti-morphism.

 We will now prove that $B_f$ is complex analytic, i.e.\ $\formalmor_f$ is an analytic anti-morphism of Lie groups.
 To this end we use the identification $\Gf (E\times \C) \cong \Gf_1 (E \times \C) \rtimes \GL (E\times \C)$ to write $(\formalmor_f (a,x+y_0,h)-y_0,h)$ as
 \begin{displaymath}
  \left(\underbrace{(x, h) + \left(\sum_{\tau \in \RT, |\tau|>1} \frac{a(\tau)h^{|\tau|}}{\sigma(\tau)}\sum_{k=0}^\infty p_k^\tau (x),0\right)}_{S_a (x,h)\coloneq}, \underbrace{\begin{pmatrix}
                                                                 \id_E & a(\bullet) f(y_0) \\0 &1
                                                                \end{pmatrix}}_{M_a\coloneq  }\right)
 \end{displaymath}
 This decomposes $\Gf (E\times \C)$ as the product $\Gf_1 (E \times \C) \rtimes \GL (E\times \C)$.
 We deduce that it suffices to prove that the maps $S_f \colon \BGC \rightarrow \Gf_1 (E\times \C), a \mapsto S_a$ and $M_f \colon \BGC \rightarrow \GL (E\times \C), a \mapsto M_a$ are complex analytic.
 However, since $\Gf_1 (E\times \C)$ is a projective limit of the Banach Lie groups $G_k$ (cf.\ \ref{setup: formdiff}), $S_a$ will be complex analytic if $\pi_k \circ \formalmor_f$ is complex analytic for each $k\in \N$ where
 \begin{displaymath}
  \pi_k \colon \Gf_1 (E\times \K) \rightarrow G_k,\hspace{8pt}  \left((x,h) + \sum_{n>1} q_n(x,h)\right) \mapsto \left((x,h) + \sum_{n>1}^k q_n(x,h)\right).
 \end{displaymath}
 This follows from the fact that $\Gf_1 (E\times \C)$ admits a global chart and its Lie algebra $\Lf (\Gf_1 (E\times \C)$ is the projective limit of the Lie algebras $\Lf (G_k)$.
 Now
 \begin{displaymath}
  \pi_k \circ \formalmor_f (a)(x,h) = \left(x + \sum_{n=1}^k\sum_{|\tau| + r = n} \frac{a(\tau)h^{|\tau|}}{\sigma(\tau)} p_r^\tau (x),h\right)
 \end{displaymath}
 is complex analytic for each $k \in \N$ as the evaluation maps $\ev_\tau \colon \BGC \rightarrow \C , a \mapsto a (\tau), \tau \in \RT$ are complex analytic (cf.\ \cite[Lemma 1.22]{BS14} or Corollary \ref{cor: ev:cont}).

 To see that also $M_f$ is complex analytic, we use that $\ev_\bullet$ is complex analytic.
 Then $M_f$ is complex analytic as one can write it as the composition of $\ev_\bullet$ with the complex analytic map $\C \rightarrow \GL(E\times \C) ,\ z \mapsto \begin{pmatrix}
                                                                               \id_E & z \\
                                                                               0 & 1
                                                                              \end{pmatrix}
 $.
 \item By part (a) the map $\formalmor_f$ is an antimorphism of complex Lie groups.
  Recall that $\BGC$ is the complexification (in the sense of Definition \ref{defn: complexification}) of the real analytic Lie group $\BGp$ and $\Gf (E\times \C)$ is a complexification of the real analytic Lie group $\Gf (F \times \R)$.
  Clearly $\formalmor_f$ restricts on $\BGp$ to $\formalmor_g$, whence $\formalmor_f$ is a complex analytic extension of $\formalmor_g$ and $\formalmor_g$ is a real analytic antimorphism of Lie groups.\qedhere
  \end{enumerate}
\end{proof}

The morphisms constructed in Proposition \ref{prop: formalmor} translates the mechanism which associates to an element of the Butcher group its B-series (i.e.\ a formal diffeomorphism) into the language of Lie theory.
We will now consider an analogous construction for the tame Butcher group.
The key idea here is however, that the morphism should not only yield a formal diffeomorphism but should incorporate information on convergence of the $B$-series.
For this reason this new construction will be more involved.
Before we discuss this in detail, we need to recall the definition of the Lie groups of germs of diffeomorphisms.
This group will replace the group of formal diffeomorphisms and serve as our new target Lie group.

\begin{setup}
 Let $E$ be a Banach space over $\K \in \{\R,\C\}$ and $y_0 \in E$.
 Consider the \emph{group of germs of analytic diffeomorphisms}
  \begin{displaymath}
   \DiffGerm (\{y_0\} , E) \coloneq \left\{\eta \middle| \substack{ \eta \text{ is a } C^\omega_\K\text{-diffeomorphism between} \\ \text{ open neighbourhoods of } \{y_0\} \text{ and } \eta (y_0)=y_0} \right\} /\sim,
  \end{displaymath}
 where $\eta_1 \sim \eta_2$ if they coincide on a common $y_0$-neighbourhood.
 Following \cite[Section 3]{MR2670675}, we model the group $\DiffGerm (\{y_0\}, E)$ on the space
 \begin{displaymath}
  \Germ (\{y_0\}, E)_{\{y_0\}} \coloneq  \left\{\gamma \middle|  \substack{ \gamma \text{ is a } C^\omega_\K\text{-map defined on an }\\ \text{ open } y_0-\text{neighbourhood} \text{ and } \gamma (y_0) =0 }\right\} /\sim .
  \end{displaymath}
 To construct the locally convex structure on the model space, define for each $n \in \N$ the Banach space
 \begin{displaymath}
   \BHol (B_{\frac{1}{n}}^E (y_0) , E)_{\{y_0\}} \coloneq \{ f \in \Hol (B_{\frac{1}{n}}^E (y_0) ,E) \mid f \text{ is bounded and } f(y_0)=0 \},
  \end{displaymath}
 whose norm is given by the supremum norm.

 Now $\Germ (\{y_0\},E)_{\{y_0\}} = \bigcup_{n \in \N}  \BHol (B_{\frac{1}{n}}^E (0) , E)_{\{y_0\}}$ is an (LB)-space, i.e.\ the locally convex inductive limit of Banach spaces.
 In fact $\Germ (\{y_0\},E)_{\{y_0\}}$ is again a Hausdorff space by \cite[Proposition 3.3.]{MR2670675}.

 Finally, the group  $\DiffGerm (\{y_0\} , E)$ admits a global chart defined via
 \begin{align*}
 \Phi \colon \DiffGerm (\{y_0\} , E)  &\rightarrow \Germ (\{y_0\},E)_{\{y_0\}} ,\\  [\eta]_\sim &\mapsto [\eta]_\sim - [\id_{E}]_\sim .
 \end{align*}
 This chart turns $\DiffGerm (\{y_0\} , E)$ into a Lie group with respect to the composition and inversion of germs of diffeomorphisms.\footnote{If $E$ is finite-dimensional these groups were studied by Pisanelli see \cite[p.116]{MR2670675} for references.}
 \end{setup}

 \begin{prop}\label{prop: nmaps:DGerm}
  For each $n \in \N$ we obtain a map
    \begin{align*}
     B_n \colon B_1^{\omega_n} (e) \cap \tBGC &\rightarrow \DiffGerm (\{(y_0,0)\}, E \times \C) ,\\ a &\mapsto [ (y,h) \mapsto (B(a,y,h),h)]_\sim,
    \end{align*}
 where $(y,h) \mapsto (B(a,y,h),h)$ is defined on an open neighbourhood of $(y_0,0) \in E \times \C$.\footnote{Recall that $e$ is the unit in $\tBGC$ and we denote by $B_1^{\omega_n} (e)$ the $1$-ball around $e$ in the Banach space $\C^{\RT_0} (\omega_n)$.}
 Let $\Phi$ be the global chart for $\DiffGerm (\{(y_0,0)\}, E \times \C)$, then there is $m (n) \in \N_0$ such that $\Phi \circ B_n$ factors through $\BHol (B_{\frac{1}{m(n)}}^{E\times \C} (y_0,0) , E \times \C)_{\{(y_0,0)\}}$.
\end{prop}

\begin{proof}
 Fix $n\in \N$ and consider the map $B_n$.
 Every $a \in B_1^{\omega_n} (e) \cap \tBGC$ satisfies $|a(\tau)| \leq \omega_n (\tau) = 2^{n|\tau|}$.
 Now Proposition \ref{prop: convergence} asserts that there are an open $y_0$-neighbourhood $V \subseteq U$ and a constant $h_0>0$, which depends only on $V$, the constant $K =2^n$ and $f$ such that the B-series converges for all $(a,y,h) \in \tBGC \times V \times B_{h_0}^\C (0)$, i.e.\
  \begin{equation}\label{repeat: Bseries}
   B_n (a) (y,h) = (B(a,y,h),h) = \left(y + \sum_{\tau \in \RT} \frac{h^{|\tau|}a(\tau)}{\sigma(\tau)} F(\tau)(y),h\right) \in E \times \C.
  \end{equation}
  In particular, we choose $m \coloneq m(n)  \in \N$ so large that $B_{\frac{1}{m}}^{E \times \C} (0,y_0) \subseteq  V \times B_{h_0}^\C (0)$.
  Restricting $B_n (a)$ to $B_{\frac{1}{k}}^{E \times \C} (y_0,0)$, we obtain a bounded map (this is already true on $V \times B_{h_0}^\C (0)$  by the estimate \eqref{est: Bseries}).

  We now proceed in several steps:\medskip

  \textbf{Step 1:} \emph{The B-series fixes $(y_0,0)$}.

 Since $f$ vanishes in $y_0$, we see that the elementary differentials $F(\tau)(y_0)$ vanish.
 This follow from Definition \ref{defn: Bseries} as the elementary differential  $F(\tau)(y_0)$  evaluates terms of nested \Frechet derivatives (which are multilinear) and the recursion formula eventually evaluates an argument of the form $f(y_0)=0$.
 Hence \eqref{repeat: Bseries} shows that $B_n (a) (y_0,0) = (y_0 , 0)$, whence the maps in the image of $B_n$ fix $(y_0,0)$.
 We deduce that
  \begin{equation}\label{eq: nummer}
   g (a) \coloneq B_n (a) (\cdot)|_{B_{\frac{1}{m}}^{E \times \C} (y_0,0)} - \id_{B_{\frac{1}{m}}^{E \times \C} (y_0,0)},
  \end{equation}
 takes $(y_0,0)$ to $(0,0)$.
 Hence the assertion of the proposition will follow if we can show that $B_n (a)$ is holomorphic and induces a diffeomorphism on a neighbourhood of $(y_0,0)$.
 \medskip

 \textbf{Step 2:} \emph{The map $B_n (a)$ is a holomorphic map on $B_{\frac{1}{m}}^{E \times \C} (y_0,0)$}.

 Define for $p \in \N$ the mapping
 \begin{displaymath}
 g_p^a (h,y) \coloneq \sum_{\substack{\tau \in \RT \\ |\tau| = p}}  \frac{h^{|\tau|}a(\tau)}{\sigma(\tau)} F(\tau)(y) \text{ for }(h,y) \in
  B_{\frac{1}{m}}^{E \times \C} (y_0,0).
  \end{displaymath}
 As each $g_p^a$ is a finite sum of a product of continuous maps, $g_p^a$ is continuous.
 Further, $f$ is holomorphic in $y$ and $F(\tau)$ is a finite composition of \Frechet derivatives of $f$, $F(\tau)$ is holomorphic in $y$ for each $\tau \in \RT$.
 Thus $g_p^a(h,\cdot)$ is holomorphic for each $h$. Moreover, fixing $y$, the map $g_p^a (\cdot, y)$ is clearly holomorphic for each $y$.
 Since mappings which are continuous and separately holomorphic are (jointly) holomorphic by a version of Hartog's theorem \cite[Proposition 8.10]{MR842435} we deduce that $g_p^a$ is a holomorphic map for each $p\in \N$.\footnote{Note that \cite{MR842435} defines holomorphic mappings as maps which are locally the uniform limit of a series of polynomials. However, by virtue of Remark \ref{rem: analytic} this concept coincides with our concept of holomorphic mappings on Banach spaces. (cf.\ also the extensive discussion in \cite{BS71b}).}

 We can now write $g(a) = \left(\sum_{p \in \N} g_p^a (y,h) , 0\right)$.
 Albeit we assumed in Proposition \ref{prop: convergence} that the parameter $h$ is a real number, the results formulated there carry over verbatim to the complex case.
 Thus we can invoke \eqref{est: Bseries} for complex $h$ with $\abs{h} \leq h_0$ to obtain
 \begin{equation}\label{eq: cp}
 \sup_{(y,h) \in B_{\frac{1}{m}}^{E \times \C} (0,y_0)}\norm{g_p (y,h)} \leq \frac 12 C \left(\frac{4eMK\abs{h}}{R}\right)^p \leq \frac 12 C \left(\frac{4eMKh_0}{R}\right)^p \equalscolon c_p.
 \end{equation}
 Now by construction of $h_0$ the series $\sum_{p \in \N} c_p$ converges.
 Thus \cite[II. \S 2 Theorem 2.1]{MR1659317} implies that the series of functions $\sum_{p \in \N} g_p^a$ converges uniformly and absolutely to $g (a)$.
 Hence $g (a)$ is holomorphic as a uniform limit of holomorphic maps (see \cite[Exercise 8.A]{MR842435}).
 We conclude from \eqref{eq: nummer} that $B_n (a)$ is holomorphic on the disc $B_{\frac{1}{m}}^{E \times \C} (y_0,0)$.\medskip

\textbf{Step 3:} \emph{$B_n (a)$ induces a $C^\omega$-diffeomorphism on some open $(y_0,0)$ neighbourhood}.

 As $\tBGC = \bigcup_{n \in \N} B_1^{\omega_n} (e) \cap \tBGC$, for each $a \in B_1^{\omega_n} (e) \cap \tBGC$ there is some $k \in \N$ with $a^{-1} \in B_1^{\omega_k} (e) \cap \tBGC$.
 Hence by the arguments above, we obtain a holomorphic map $B_k (a^{-1})$ which is defined on some open $(y_0,0)$-neighbourhood and takes $(y_0,0)$ to itself.
 Without loss of generality we may assume that $k \geq n$. Hence the definition of the maps $B_n$ and $B_k$ shows that $B_n (a) = B_k (a)$ (as germs of holomorphic mappings).
 Composing $B_k (a^{-1})$ and $B_n (a)$ (which is possible on some open $(y_0,0)$-neighbourhood!), \eqref{eq: Bseries:insert} shows $B_k (a^{-1}) \circ B_n (a) = B_k (e) = \id$.
 Hence $B_n (a)$ induces a $C^\omega$-diffeomorphism of suitable open $(y_0,0)$-neighbourhoods.

 In conclusion, $B_n$ makes sense as a mapping to $\DiffGerm (\{(y_0,0)\}, E \times \C)$ and $\Phi \circ B_n$ factors through $\BHol (B_{\frac{1}{m}}^{E\times \C} (y_0,0) , E \times \C)_{\{(y_0,0)\}}$.
\end{proof}

We are now in a position to prove the main result of this section. To this end we recall again that $\tBGC = \bigcup_{n \in \N}  B_1^{\omega_n} (e) \cap \tBGC$.
On each step of the union we have constructed a map into germs of diffeomorphisms and these maps glue to a morphism of Lie groups.

\begin{thm}\label{thm: LGP:Mor}
 Let $E$ be a complex Banach space and $y_0 \in E$ fixed.
 Consider a holomorphic map $f \colon E \supseteq U \rightarrow E$ on a $y_0$-neighbourhood $U$ with $f(y_0) = 0$.
 Then
 \begin{enumerate}
  \item The assignment
 \begin{align*}
  B_f \coloneq \bigcup_{n \in \N} B_n \colon \tBGC &\rightarrow \DiffGerm (\{(y_0,0)\}, E \times \C),\\ a &\mapsto B_n (a) \text{ if } a \in B_1^{\omega_n} (e),
 \end{align*}
 is a complex analytic antimorphism of Lie groups (i.e.\ it is a Lie group morphism to the opposite group)
 \end{enumerate}
 Let $g \colon F \supseteq W \rightarrow F$ be a real analytic map such that $E$ is a complexification of $F$ and $f \colon E \supseteq U \rightarrow E$ its holomorphic extension, i.e.\ $f|_{U \cap F} = g$.
 \begin{enumerate}
  \item[\upshape (b)] The morphism $B_f$ restricts to a real analytic antimorphism of Lie groups
  \begin{displaymath}
  B_{g} \colon \tBG \rightarrow \DiffGerm (\{(y_0,0)\}, F \times \R) ,\quad a \mapsto B_f (a)|_{\dom B_f (a) \cap (F \times \R)}.
 \end{displaymath}
 \end{enumerate}
\end{thm}

\begin{proof}
 \begin{enumerate}
  \item By definition of the equivalence relation $\sim$ on $\DiffGerm (\{(y_0,0)\}, E \times \C)$ two maps coincide if their germs are equivalent.
  Hence we deduce from the construction $B_n (a) = B_{m} (a)$ for all $n,m \in \N$ with $m \leq n$ and $a \in B_1^{\omega_n} (e) \cap \tBGC$.
  Thus $B_f$ makes sense as a map.
  Then $B_f$ induces  a group antimorphism as \eqref{eq: Bseries:insert} implies
    \begin{displaymath}
     B_f (b) \circ B_f (a)  = B_f (a\cdot b) \text{ for all } a,b \in \tBGC.
    \end{displaymath}
  We now prove that $B_f$ is continuous.
  Since $\tBGC$ is an affine subspace of a Silva space, it suffices to check continuity on each step of the inductive limit, i.e.\ we have to prove that the maps $B_n$ are continuous.
  As $\DiffGerm  (\{(y_0,0)\}, E \times \C)$ admits a global chart $\Phi$ it suffices to check that $\Phi \circ B_n$ is continuous for each $n \in \N$.

  Now for fixed $n\in \N$, Proposition \ref{prop: nmaps:DGerm} asserts that there is $m \in \N$ such that $\Phi \circ B_n$ factors through $X_m \coloneq \BHol (B_{\frac{1}{m}}^{E\times \C} (y_0,0) , E \times \C)_{\{(y_0,0)\}}$.
  In the following, we will abbreviate $O_m \coloneq B_{\frac{1}{m}}^{E\times \C} (y_0,0)$.
  As the model space $\Germ (\{(y_0,0)\},E \times \C)_{\{(y_0,0)\}}$ is the inductive limit of the Banach spaces $X_m$, it suffices thus to prove that $\Phi \circ B_n$ is continuous as a map to $X_m$.
  Now we have reduced the problem to prove that the mapping
    \begin{displaymath}
     h_n \colon B_1^{\omega_n} (e) \cap \tBGC \rightarrow X_m, a \mapsto h_n(a) \coloneq B_n (a)|_{O_m} - \id_{O_m},
    \end{displaymath}
 is continuous. To see this, pick a sequence $(a_k)_{k \in \N} \subseteq B_1^{\omega_n} (e) \cap \tBGC$ which converges to $a$.
 Then we compute -- using the absolute convergence of the series which was established in Step 2 of the proof of Proposition \ref{prop: nmaps:DGerm} -- the estimate
 \begin{equation}  \begin{aligned}
 & \hphantom{\stackrel{\eqref{eq: Bseries}}{\leq}}\sup_{(y,h) \in O_m} \norm{h_n (a_k) - h_n (a)} =  \sup_{(y,h) \in O_m} \norm{\sum_{\tau \in \RT} \frac{h^{|\tau|} (a_k (\tau) - a(\tau))}{\sigma (\tau)} F (\tau) (y)} \\
	  &  \stackrel{\hphantom{\text{Proposition }\ref{prop: convergence}}}{\leq} \sup_{(y,h) \in O_m} \sum_{q = 1}^\infty \norm{\sum_{\substack{\tau \in \RT \\ |\tau| = q}} \frac{h^{q} (a_k (\tau) - a(\tau))}{\sigma (\tau)} F (\tau) (y)}  \\
	  &\stackrel{\text{Proposition }\ref{prop: convergence}}{\leq} \frac{1}{2} \norm{a_k - a}_{\omega_n} \underbrace{\sum_{q\in \N} \left( \frac{4eMKh_0}{R}\right)^q}_{< \infty \text{ by choice of } h_0 \text{ cf. } \eqref{eq: cp}}.
	  \end{aligned}\label{est: cont:stepmap}
 \end{equation}
 Now $a_k \rightarrow a$ in $B_1^{\omega_n} (e) \cap \tBGC$ \eqref{est: cont:stepmap} entails that $h_n$ is continuous.
 As $n \in \N$ was arbitrary, we deduce that also $B_f$ is continuous.

 Finally, both Lie groups admit global charts $\Phi$ and $\Psi \colon \tBGC \rightarrow \Msp{\C} , a \mapsto a-e$, respectively.
 Composing with these charts, we obtain for $a,b \in \Msp{\C}$ and $z \in \C$ the identity
  \begin{align*}
   \Phi \circ B_f \circ \Psi^{-1} (a+zb) &= \sum_{\tau \in \RT} \frac{h^{|\tau|} (a (\tau) + zb(\tau))}{\sigma (\tau)} F (\tau) (y) \\
   &= \sum_{\tau \in \RT} \frac{h^{|\tau|} a (\tau)}{\sigma (\tau)} F (\tau) (y) + z\sum_{\tau \in \RT} \frac{h^{|\tau|}  b(\tau))}{\sigma (\tau)} F (\tau) (y)\\
   &= \Phi \circ B_f \circ \Psi^{-1} (a) + z \Phi \circ B_f \circ \Psi^{-1} (b).
  \end{align*}
 Hence these charts take $B_f$ to a continuous linear, whence holomorphic map.
 Thus $B_f \colon \tBGC \rightarrow (\DiffGerm (\{(y_0,0)\}, E \times \C) ,\circ)$ is an antimorphism of complex Lie groups.
 \item By part (a) the map $B_f$ is a antimorphism of complex Lie groups.
  Corollary \ref{cor: LGP:compl} asserts that $\tBGC$ is a complexification of the real analytic Lie group $\tBG$.
  Moreover, there is a canonical inclusion
    \begin{displaymath}
     \text{inc} \colon \DiffGerm (\{(y_0,0)\}, F \times \R) \rightarrow \DiffGerm (\{(y_0,0)\}, E \times \C). [\eta]_\sim \mapsto [\tilde{\eta}]_\sim ,
    \end{displaymath}
 where $\tilde{\eta}$ is a map which restricts to $\eta$ on the intersection of its domain with $F \times \R$.
 This inclusion turns $\DiffGerm (\{(y_0,0)\}, E \times \C)$ into a complexification of the real analytic Lie group $\DiffGerm (\{(y_0,0)\}, F \times \R)$ (see \cite{MR2670675} and \cite[Corollary 15.11]{hg2007}).
 Now it is easy to see that the antimorphism $B_f$ takes $\tBG$ to the image of $\text{inc}$.
 Hence we deduce from the properties of the complexifications and $\text{inc} \circ B_g = B_f|_{\tBG}$ that $B_g$ is a real analytic antimorphism of Lie groups.  \qedhere
\end{enumerate}
\end{proof}

\begin{rem}
 \begin{enumerate}
  \item The Lie group antimorphism from part (b) of Theorem \ref{thm: LGP:Mor} yields germs of diffeomorphisms which take the step size $h$ as a real parameter.
  Hence these mappings are more closely related to the way B-series are treated in numerical analysis.
  Note however, that we needed complex parameters $h$ in the proof as the real case follows by a complexification argument from the complex case.
  \item Our proof depends in a crucial way on the estimates in Proposition \ref{prop: convergence}.
  In particular, this implies that the construction can not be lifted to yield a similar antimorphism from the Butcher group to $\DiffGerm (\{(y_0,0)\}, E \times \R)$. Similarly, the construction does not yield an antimorphism for the complex Butcher group.
 \end{enumerate}
\end{rem}

 We will now discuss the relation of the Lie group morphisms constructed in Proposition \ref{prop: formalmor} and Theorem \ref{thm: LGP:Mor}.
 To this end, let us first construct a Lie group morphism from the group of germs of diffeomorphisms to the group of formal diffeomorphisms.

 \begin{setup}
 Let again $E$ be a complex Banach space which is the complexification of a real Banach space $F$.
 For a germ $[\eta] \in \DiffGerm (\{(y_0,0)\}, E \times \C)$ we let $\text{Tay} ([\eta])$ be the Taylor series at $(y_0,0)$, i.e.\ we take the Taylor series of any representative of $[\eta]$ at $(y_0,0)$ and note that this is independent of choice of representative. This construction yields an injective group morphism
 \begin{align*}
  \Gamma \colon \DiffGerm (\{(y_0,0)\}, E \times \C) &\rightarrow \Gf (E\times \C),\\ [\eta] &\mapsto \left[ (x,h) \mapsto  \text{Tay} (\eta) (x+y_0,h) - (y_0 ,0)\right].
 \end{align*}
 Composing with charts, an easy inductive limit argument together with the identification $\Gf (E\times \C) \cong \Gf_1 (E\times \C) \times \GL(E\times \C)$ shows that $\Gamma$ is even smooth, i.e.\ an injective morphism of Lie groups.
 Similarly one constructs an injective Lie group morphism $\Gamma_\R \colon \DiffGerm (\{(y_0,0)\}, F \times \R) \rightarrow \Gf (F\times \R)$.
 \end{setup}

  Assume now that $y_0 \in F$ and $f$ and $g$ are as in the statement of Theorem \ref{thm: LGP:Mor}.
  One concludes with Proposition \ref{prop: formalmor} and Theorem \ref{thm: LGP:Mor} that the following diagram of Lie group morphisms and antimorphisms commutes
  \begin{equation}\label{diag: huge} \begin{aligned}
  \begin{xy}
  \xymatrix{
   &   \DiffGerm (\{(y_0,0)\}, F \times \R) \ar[rr]^-{\Gamma_\R} \ar[dd]|!{[d];[r]}\hole  &  &  \Gf (F\times \R) \ar[dd] & \ar[dd]^{\rotatebox{90}{ \text{ \footnotesize Complexification}}} \\
   \tBG \ar[rr]^(.65){\alpha_\R} \ar[dd] \ar[ru]^{B_g}  &  &  \BGp \ar[dd] \ar[ru]^{\formalmor_g}  &  &\\
   &   \DiffGerm (\{(y_0,0)\}, E \times \C) \ar[rr]|!{[d];[r]}\hole^{\Gamma}  &  &  \Gf (E\times \C) & \\
   \tBGC \ar[rr]^{\alpha} \ar[ru]^{B_f}  &  & \BGC \ar[ru]^{\formalmor_f}  & &
  }
\end{xy}
\end{aligned}
\end{equation}
 Here the vertical arrows in \eqref{diag: huge} represent the canonical mappings into the respective complexifications (cf.\ diagram \eqref{eq: relation}).

 The commutativity in the above diagram provides a new justification and concretization of a proof technique often used for B-series.
 We illustrate this by discussing the following example.

\begin{ex}[Preservation of quadratic integrals]
 In \cite[Theorem VI.7.1]{HLW2006} a condition is provided which ensures that a B-series method preserves a quadratic integral.
 Let us recall some details: Fix an analytic map $f \colon E \rightarrow E$ on a Banach space and assume that there is a quadratic function $Q\colon E \to \C$ satisfying $Q'(y) f(y)=0$ for all $y\in E$.

 Now one is interested in conditions on the element $a$ in the Butcher group such that the associated B-series method preserves the quadratic integral.
 Using the morphisms we introduced in the present section, the B-series method preserves the quadratic integral with respect to $Q$ if
 \begin{equation}\label{eq: first int}
  Q \left(\underbrace{ y + \sum_{\tau \in \RT} \frac{h^{|\tau|}}{\sigma (\tau)}a(\tau) F_f (\tau)(y) }_{\text{pr}_E \circ\formalmor_f(a)(y,h)}\right)= Q (y),
 \end{equation}
 where $\text{pr}_E \colon E\times \C \to E$ is the canonical projection.
 Fix $y_0 \in E$. Formally we can rewrite \eqref{eq: first int} as $Q\circ \text{pr}_E \circ \formalmor_f(a) = Q \circ \text{pr}_E$, where the composition means composition of the formal series $Q\colon E \to \C$ around $y_0$ with the formal series $\text{pr}_E \circ \formalmor_f(a) \colon E \times \C \to E$ around $(y_0,0)$, resulting in a formal series  $Q\circ \text{pr}_E \circ \formalmor_f(a)\colon E \times \C \to \C$ around $(y_0, 0)$.

 As $\formalmor_f \circ \alpha = \Gamma \circ B_f$, we see that for an element $a = \alpha (a')$ of the tame Butcher group, the equality \eqref{eq: first int} pulls back to an equality
 $$Q \circ \text{pr}_E \circ B_f(a') = Q \circ \text{pr}_E,$$
where both sides are germs of smooth function $ E \times \C \to \C$ around $(y_0,0)$.

In other words: Let $a$ be an element in the tame Butcher group, $f$ an analytical vector field and $y_0\in E$ satisfying the assumptions of \cite[Theorem VI.7.1]{HLW2006}.
Further we let $Q$ be a quadratic function satisfying $Q'(y) f(y)=0$ for all $y \in E$, then there exists a neighbourhood $U$ of $(y_0, 0)$ such that the B-series method converges and satisfies $Q (B(a,y,h)) = Q(y)$ for all $(y,h) \in U$.
 \end{ex}

 \section*{Acknowledgements}
The research on this paper was partially supported by the projects \emph{Topology in Norway} (Norwegian Research Council project 213458) and \emph{Structure Preserving Integrators, Discrete Integrable Systems and Algebraic Combinatorics} (Norwegian Research Council project 231632). We thank the anonymous referee for many insightful comments which led to several improvements of the article.
The authors would also like to thank Helge Gl\"{o}ckner and Rafael Dahmen for helpful discussions on the subject of this work.
Moreover, we thank Ernst Hairer for providing Proposition \ref{prop: convergence}.

\begin{appendix}
 \section{Locally convex differential calculus and manifolds}\label{App: lcvx:diff}

 Basic references for differential calculus in locally convex spaces are \cite{hg2002a,keller1974}.
 For the reader's convenience, we recall various definitions and results:

\begin{defn}\label{defn: deriv}
 Let $\K \in \{\R,\C\}$, $r \in \N \cup \{\infty\}$, $E$, $F$ locally convex $\K$-vector spaces and $U \subseteq E$ open.
 Moreover we let $f \colon U \rightarrow F$ be a map.
 If it exists, we define for $(x,h) \in U \times E$ the (real or complex) directional derivative
 $$df(x,h) \coloneq D_h f(x) \coloneq \lim_{t\rightarrow 0} t^{-1} (f(x+th) -f(x)).$$
 The map $f$ is called $C^r_\K$ if the iterated directional derivatives
    \begin{displaymath}
     d^{(k)}f (x,y_1,\ldots , y_k) \coloneq (D_{y_k} D_{y_{k-1}} \cdots D_{y_1} f) (x),
    \end{displaymath}
 exist for all $k \in \N_0$ such that $k \leq r$, $x \in U$ and $y_1,\ldots , y_k \in E$ and define continuous maps $d^{(k)} f \colon U \times E^k \rightarrow F$.
 If it is clear which $\K$ is meant, we simply write $C^r$ for $C^r_\K$.
 If $f$ is $C^\infty_\C$, $f$ is called \emph{holomorphic} and if $f$ is $C^\infty_\R$ we say that $f$ is \emph{smooth}.
\end{defn}

\begin{rem}\label{rem: analytic}
  A map $f \colon U \rightarrow F$ is of class $C^\infty_\C$ if and only if it \emph{complex analytic} i.e.,
  if $f$ is continuous and locally given by a series of continuous homogeneous polynomials (cf.\ \cite[Proposition 7.7]{MR2069671}).
 Then $f$ is said to be of class $C^\omega_\C$.
\end{rem}

For the readers convenience we prove the following lemma which seems to be part of the mathematical folklore.
\begin{lem}\label{lem:folklore}
 Let $U$ be an open subset of a complex locally convex space $E$ and $F$ be a complex locally convex space which is sequentially complete. Consider a set $\Lambda \subseteq L(F,\C)$ of complex linear functionals which separates the points on $F$.\footnote{i.e.\ for each $x\in F$ there is $\lambda \in \Lambda$ with $\lambda (x) \neq 0$.}. If a map $f \colon U \rightarrow F$ is continuous and
  \begin{displaymath}
   \lambda \circ f \colon U \rightarrow \C
  \end{displaymath}
 is complex analytic for each $\lambda \in \Lambda$, then $f$ is complex analytic.
\end{lem}

\begin{proof}
 For all $x \in U$ and $y \in E$ and small enough $r >0$ we have for $|z|<r$ the equality
  \begin{displaymath}
   f(x+zy)=\frac{1}{2\pi i} \int_{|\zeta| = r} \frac{f(x+\zeta y)}{\zeta-z} d\zeta
  \end{displaymath}
 holds, where the integral on the right hand side is a weak integral.
 To see this note that it suffices to test with the functionals $\lambda \in \Lambda$ (as $\Lambda$ is point separating) and $z \mapsto \lambda (f(x+zy))$ is holomorphic.
 Hence by \cite[Theorem 3.1]{BS71b} the map $z \mapsto f(x+zy)$ is complex analytic and thus $f$ is G\^{a}teaux-analytic by \cite[Proposition 5.5]{BS71b}.
 As $f$ is continuous and G\^{a}teaux-analytic, \cite[Corollary 6.1]{BS71b} entails that $f$ is complex analytic.
\end{proof}

\begin{defn}
 Let $E$, $F$ be real locally convex spaces and $f \colon U \rightarrow F$ defined on an open subset $U$.
 A map $f$ is said to be \emph{real analytic} (or $C^\omega_\R$) if $f$ extends to a $C^\infty_\C$-map $\tilde{f}\colon \tilde{U} \rightarrow F_\C$ on an open neighbourhood $\tilde{U}$ of $U$ in the complexification $E_\C$.\footnote{This property was first defined by Milnor in \cite{milnor1983} for sequentially complete spaces (see \cite{hg2002a} for the general case)}
\end{defn}

For $\K \in \{\R,\C\}$ and $r \in \N_0 \cup \{\infty, \omega\}$ the composition of $C^r_\K$-maps (if possible) is again a $C^r_\K$-map (cf. \cite[Propositions 2.7 and 2.9]{hg2002a}).

\begin{defn} Fix a Hausdorff topological space $M$ and a locally convex space $E$ over $\K \in \{\R,\C\}$.
An ($E$-)manifold chart $(U_\kappa, \kappa)$ on $M$ is an open set $U_\kappa \subseteq M$ together with a homeomorphism $\kappa \colon U_\kappa \rightarrow V_\kappa \subseteq E$ onto an open subset of $E$.
Two such charts are called $C^r$-compatible for $r \in \N_0 \cup \{\infty,\omega\}$ if the change of charts map $\nu^{-1} \circ \kappa \colon \kappa (U_\kappa \cap U_\nu) \rightarrow \nu (U_\kappa \cap U_\nu)$ is a $C^r$-diffeomorphism.
A $C_\K^r$-atlas of $M$ is a set of pairwise $C^r$-compatible manifold charts, whose domains cover $M$. Two such $C^r$-atlases are equivalent if their union is again a $C^r$-atlas.

A \emph{locally convex $C^r$-manifold} $M$ modelled on $E$ is a Hausdorff space $M$ with an equivalence class of $C^r$-atlases of ($E$-)manifold charts.
\end{defn}

 Direct products of locally convex manifolds, tangent spaces and tangent bundles as well as $C^r$-maps of manifolds may be defined as in the finite dimensional setting (cf.\ e.g.\ \cite{MR2069671,neeb2006}).

\begin{defn}
A \emph{Lie group} is a group $G$ equipped with a $C^\infty$-manifold structure modelled on a locally convex space, such that the group operations are smooth.
If $G$ has a $\K$-analytic manifold structure and the group operations are in addition ($\K$-)analytic we call $G$ a \emph{($\K$)-analytic} Lie group.
\end{defn}

Recall the following criterion which allows differentiability of mappings on Silva spaces to be established:

\begin{prop}[{\cite[Proposition 4.5]{hg2011}}]\label{prop: silva}
 Consider the locally convex direct limit $E = \bigcup_{n \in \N} = E_n$ of Banach spaces $E_1 \subseteq E_2 \subseteq \cdots$ over $\K \in \{\R,\C\}$.
 Assume that the inclusion maps $E_n \rightarrow E_{n+1}$ are compact for each $n \in \N$, i.e.\ $E$ is a Silva space.

 \begin{enumerate}
  \item Then $E$ is Hausdorff and the locally convex direct limit topology coincides with the direct limit topology on $E$ as a direct limit of the $E_n, n \in \N$.
 \end{enumerate}
 Let $U_1 \subseteq U_2 \subseteq \cdots$ be an ascending sequence of open sets $U_n \subseteq E_n , n \in \N$, and $U \coloneq \bigcup_{n \in \N} U_n$.
 Fix $r \in \N_0 \cup \{\infty\}$ together with a Hausdorff locally convex space $F$ and let $f \colon U \rightarrow F$ be a map such that $f|_{U_n}$ is of class $C^r_\K$ for each $n \in \N$.
 \begin{enumerate}
 \item[\textup{(b)}]  Then $U$ is open in $E$ and $f$ is a $C^r_\K$ map.
 \end{enumerate}
\end{prop}

\section{Estimates and bounds for tree maps}\label{App: est:TM}

In this section we compute several technical auxiliary results needed to prove theorems in the main part.
To obtain upper bounds for products and inverses of (exponentially bounded) tree maps, we need the some estimates.

 \begin{setup}[Estimates for sets associated to a tree]
 Fix a rooted tree $\tau \in \RT$.
 Note that the number of nodes is an upper bound for the number of edges in the tree.
 Hence the number of splittings of a tree (see Definition \ref{defn: split}) is bounded by
 \begin{displaymath}
  |\SP{\tau}| = |\tau|+1  \text{ and } |\SP{\tau}_1| = |\tau| -1 .
 \end{displaymath}
 Moreover, the following estimates on the size of sets associated to ordered subtrees and partitions of a tree (see Notation \ref{nota: OST:P}) follow from identifying them as subsets of the power set of all nodes of $\tau$:
 \begin{align}
  |\OST (\tau)| &\leq 2^{|\tau|}  \text{  and}& \#(\tau \setminus s) &\leq |\tau| - |s_\tau|& \forall s &\in \OST (\tau),\label{eq: OST:est}\\
  |\pP (\tau)| &\leq 2^{|\tau|}  \text{  and} &\#(\tau \setminus p) &\leq |\tau|& \forall  p &\in \pP(\tau).  \label{eq: Part:est}
 \end{align}
\end{setup}

\begin{lem}\label{lem: tree:diff:est}
 Let $\tau$ be a rooted tree, $k \in \N$, $R \geq 1$ and $0< \ve <1$.
 Fix $a \in B_R^{\omega_k} (0)$ and $c \in B_\ve^{\omega_k} (0)$.
 If $s$ is either an element of $\OST (\tau)$ or a partition in $\pP (\tau)$ we obtain the following estimate
  \begin{displaymath}
    |(a +c) (\tau \setminus s) - a(\tau\setminus s)| \leq  \frac{2^{|\tau|}R^{|\tau|-1} \ve}{\prod_{\theta \in \tau \setminus s}\omega_k (\theta)}.
  \end{displaymath}
\end{lem}

\begin{proof}
 Let us denote the power set of the forest $\tau \setminus s$ by $\text{Pow} (\tau \setminus s)$.
 Using that $R\geq 1$ and $\ve < 1$ we obtain the following estimates:
 \begin{align*}
   &|(a +c) (\tau \setminus s) - a(\tau\setminus s)| = \left|\prod_{\theta \in \tau \setminus s} (a(\theta) + c(\theta)) - \prod_{\theta \in \tau \setminus s} a(\theta)\right| \\
	  =& \left|\sum_{\substack{P \in \text{Pow} (\tau \setminus s)\\ P \neq \tau \setminus s}} \prod_{\theta \in P} a(\theta) \prod_{\delta \in (\tau \setminus s) \setminus P} c(\delta)\right|
	  \leq \sum_{\substack{P \in \text{Pow} (\tau \setminus s)\\ P \neq \tau \setminus s}} \prod_{\theta \in P} \frac{R}{\omega_k (\theta)} \prod_{\delta \in (\tau \setminus s) \setminus P} \frac{\varepsilon}{\omega_k (\delta)}\\
	  \leq& \sum_{\substack{P \in \text{Pow} (\tau \setminus s)\\ P \neq \tau \setminus s}} \frac{\ve R^{|\tau|-1}}{\prod_{\theta \in \tau \setminus s }\omega_k (\theta)} \leq \frac{2^{|\tau|}R^{|\tau|-1} \ve}{\prod_{\theta \in \tau \setminus s }\omega_k (\theta)}.
 \end{align*}
 For the last inequality one uses \eqref{eq: OST:est} or \eqref{eq: Part:est} to see that $|\text{Pow} (\tau \setminus s)| \leq 2^{\#(\tau \setminus s)} \leq 2^{|\tau|}$.
\end{proof}

\begin{lem}\label{lem: estimates1}
 Let $R \geq 1$, $k \in \N$ and $a ,b \in B^{\omega_k}_R (e) \cap \tBGC$.
 Furthermore, fix $1 > \ve > 0$ with $\norm{a}_{\omega_k} + \norm{b}_{\omega_k} + 2 \ve < R$  and $c, d \in B^{\omega_k}_\ve (0) \cap \Msp{\C}$.
 Then $a + c , b +c \in \tBGC$ and for evaluating their product in $\tau \in \RT$, we obtain the following estimates:
 \begin{align}\label{eq: est1}
  |\ev_\tau [(a + c) \cdot (b + d) - a \cdot b]| \leq 2\ve \frac{(4R)^{|\tau|}}{\omega_k (\tau)}.
 \end{align}
\end{lem}

\begin{proof}
Since $\tBGC$ is the affine subspace $e+ \Msp{\C}$ (see \ref{lem: aff:subs}), the elements $a+c$ and $b+c$ are contained in $\tBGC$.
Note that $\norm{e}_{\omega_k} = \sup_{\tau \in \RT_0} \abs{e(\tau) \omega_k (\tau)} = \abs{ e(\emptyset) \omega_k (\emptyset)} = 1 \leq R$.
Hence $B_R^{\omega_k} (e) \subseteq B_{2R}^{\omega_k} (0)$.
We now compute:
 \begin{align*}
   &|\ev_\tau [ (a + c) \cdot (b + d) - a \cdot b]|\\
   &= \left| \sum_{s \in \OST (\tau)} ( (b + d) (s_{\tau}) (a +c) (\tau\setminus s) - b(s_\tau) a(\tau\setminus s))  \right|\\
   &= \left|\sum_{s \in \OST (\tau)} b(s_\tau) ((a +c) (\tau \setminus s) - a(\tau\setminus s)) + d(s_{\tau})(a+c) (\tau \setminus s) \right|  \\
   &\leq \sum_{s \in \OST(\tau)} \left(\underbrace{\frac{2R |(a+c) (\tau\setminus s) - a(\tau\setminus s)|}{\omega_k (s_{\tau})}}_{\equalscolon S_1}
   + \underbrace{\frac{\ve \prod_{\theta \in \tau \setminus s} |a(\theta) + c (\theta)|} {\omega_k (s_{\tau})}}_{\equalscolon S_2}\right).
  \end{align*}
 We consider the summands $S_1$ and $S_2$ separately.
From Lemma \ref{lem: tree:diff:est} we conclude
 \begin{displaymath}
  S_1 \leq \frac{2R}{\omega_k (s_{\tau})} \cdot  \frac{2^{|\tau|}(2R)^{|\tau|-1} \ve}{\prod_{\theta \in \tau \setminus s}\omega_k (\theta)} \stackrel{\text{Remark \ref{rem: weights:mult}}}{=}  \frac{(4R)^{|\tau|} \ve}{\omega_k(\tau)}.
 \end{displaymath}
To obtain a bound for $S_2$ recall that $\norm{a}_{\omega_k} + \ve < 2R$.
Hence, $|a (\theta) + c(\theta)| \leq \tfrac{2R}{\omega_k (\theta)}$ holds for all $\theta \in \RT$.
Using again that $\omega_k$ is multiplicative, we obtain $|S_2| \leq \frac{\ve (2R)^{|\tau|}}{\omega_k (\tau)}$.

Plug the estimates for $S_1$ and $S_2$ into the formula.
As $|\OST(\tau)| \leq 2^{|\tau|}$ holds by \eqref{eq: OST:est}, the formula then yields the desired estimate.
\end{proof}
\end{appendix}



\begin{thebibliography}{99}
 \def\polhk#1{\setbox0=\hbox{#1}{\ooalign{\hidewidth
  \lower1.5ex\hbox{`}\hidewidth\crcr\unhbox0}}}
\providecommand{\url}[1]{\texttt{#1}}
\providecommand{\urlprefix}{URL }

\providecommand{\eprint}[2][]{\url{#2}}

\bibitem[Bas64]{MR0177277}
Bastiani, A.
\newblock \emph{Applications diff\'erentiables et vari\'et\'es
  diff\'erentiables de dimension infinie}.
\newblock J. Analyse Math. \textbf{13} (1964):1--114

\bibitem[BDS15]{1501.05221v3}
Bogfjellmo, G., Dahmen, R. and Schmeding, A.
\newblock \emph{{Character groups of Hopf algebras as infinite-dimensional Lie
  groups}} 2015. \newblock to appear in Ann. Inst. Fourier (Grenoble),
\newblock cf.\ \eprint{arXiv:1501.05221v3}

\bibitem[BGN04]{MR2069671}
Bertram, W., Gl{\"o}ckner, H. and Neeb, K.-H.
\newblock \emph{Differential calculus over general base fields and rings}.
\newblock Expo. Math. \textbf{22} (2004)(3):213--282.

\bibitem[Bro04]{Brouder-04-BIT}
Brouder, C.
\newblock \emph{{Trees, renormalization and differential equations}}.
\newblock BIT Num. Anal. \textbf{44} (2004):425--438

\bibitem[BS71a]{BS71a}
Bochnak, J. and Siciak, J.
\newblock \emph{Polynomials and multilinear mappings in topological vector
  spaces}.
\newblock Studia Math. \textbf{39} (1971):59--76

\bibitem[BS71b]{BS71b}
Bochnak, J. and Siciak, J.
\newblock \emph{Analytic functions in topological vector spaces}.
\newblock Studia Math. \textbf{39} (1971):77--112

\bibitem[BS15]{BS14}
Bogfjellmo, G. and Schmeding, A.
\newblock \emph{The lie group structure of the Butcher group}.
\newblock Found. Comput. Math. (2015):1--33.
\newblock DOI: 10.1007/s10208-015-9285-5.

\bibitem[But72]{Butcher72}
Butcher, J.~C.
\newblock \emph{An algebraic theory of integration methods}.
\newblock Math. Comp. \textbf{26} (1972):79--106

\bibitem[CHV10]{CHV2010}
Chartier, P., Hairer, E. and Vilmart, G.
\newblock \emph{Algebraic structures of {B}-series}.
\newblock Found. Comput. Math. \textbf{10} (2010)(4):407--427.

\bibitem[Dah10]{MR2670675}
Dahmen, R.
\newblock \emph{Analytic mappings between {LB}-spaces and applications in
  infinite-dimensional {L}ie theory}.
\newblock Math. Z. \textbf{266} (2010)(1):115--140.

\bibitem[Dah11]{dahmen2011}
Dahmen, R.
\newblock \emph{{D}irect limit constructions in infinite dimensional {L}ie
  theory}.
\newblock Ph.D. thesis, {University of Paderborn} 2011.
\newblock \eprint{urn:nbn:de:hbz:466:2-239}

\bibitem[Flo71]{MR0287271}
Floret, K.
\newblock \emph{Lokalkonvexe {S}equenzen mit kompakten {A}bbildungen}.
\newblock J. Reine Angew. Math. \textbf{247} (1971):155--195

\bibitem[FW68]{FW68}
Floret, K. and Wloka, J.
\newblock \emph{Einf\"uhrung in die {T}heorie der lokalkonvexen {R}\"aume}.
\newblock Lecture Notes in Mathematics, No. 56 (Springer-Verlag, Berlin-New
  York, 1968)

\bibitem[Gl{\"o}02]{hg2002a}
Gl{\"o}ckner, H.
\newblock \emph{Infinite-dimensional {L}ie groups without completeness
  restrictions}.
\newblock In \emph{Geometry and analysis on finite- and infinite-dimensional
  {L}ie groups ({B}\polhk edlewo, 2000)}, Banach Center Publ., vol.~55, pp.
  43--59 (Polish Acad. Sci., Warsaw, 2002)

\bibitem[Gl{\"o}07]{hg2007}
Gl{\"o}ckner, H.
\newblock \emph{Direct limits of infinite-dimensional {L}ie groups compared to
  direct limits in related categories}.
\newblock J. Funct. Anal. \textbf{245} (2007)(1):19--61.

\bibitem[Gl{\"o}11]{hg2011}
Gl{\"o}ckner, H.
\newblock \emph{Direct limits of infinite-dimensional {L}ie groups}.
\newblock In \emph{Developments and trends in infinite-dimensional {L}ie
  theory}, Progr. Math., vol. 288, pp. 243--280 (Birkh{\"a}user Boston, Inc.,
  Boston, MA, 2011).

\bibitem[Gl{\"o}15a]{hg2015}
Gl{\"o}ckner, H.
\newblock \emph{{Fundamentals of submersions and immersions between
  infinite-dimensional manifolds}} 2015.
\newblock \eprint{arXiv:1502.05795}

\bibitem[Gl{\"o}15b]{HG15reg}
Gl{\"o}ckner, H.
\newblock \emph{{Regularity properties of infinite-dimensional Lie groups, and
  semiregularity}} 2015.
\newblock \eprint{arXiv:1208.0715v3}

\bibitem[HL97]{HL1997}
Hairer, E. and Lubich, C.
\newblock \emph{The life-span of backward error analysis for numerical
  integrators}.
\newblock Numer. Math. \textbf{76} (1997)(4):441--462.

\bibitem[HLW06]{HLW2006}
Hairer, E., Lubich, C. and Wanner, G.
\newblock \emph{{Geometric Numerical Integration}}, Springer Series in
  Computational Mathematics, vol.~31 (Springer Verlag, $^2$2006)

\bibitem[HSTH01]{MR1878717}
Hirai, T., Shimomura, H., Tatsuuma, N. and Hirai, E.
\newblock \emph{Inductive limits of topologies, their direct products, and
  problems related to algebraic structures}.
\newblock J. Math. Kyoto Univ. \textbf{41} (2001)(3):475--505

\bibitem[IR90]{MR1070716}
Ireland, K. and Rosen, M.
\newblock \emph{A classical introduction to modern number theory}, Graduate
  Texts in Mathematics, vol.~84 (Springer-Verlag, New York, 1990), second edn.

\bibitem[Kel74]{keller1974}
Keller, H.
\newblock \emph{{Differential Calculus in Locally Convex Spaces}}.
\newblock Lecture Notes in Mathematics 417 (Springer Verlag, Berlin, 1974)

\bibitem[KM97]{KM97}
Kriegl, A. and Michor, P.~W.
\newblock \emph{{The convenient setting of global analysis}}, {Mathematical
  Surveys and Monographs}, vol.~53 (AMS, 1997)

\bibitem[KR01]{MR1828283}
Kamran, N. and Robart, T.
\newblock \emph{A manifold structure for analytic isotropy {L}ie pseudogroups
  of infinite type}.
\newblock J. Lie Theory \textbf{11} (2001)(1):57--80

\bibitem[Lan99]{MR1659317}
Lang, S.
\newblock \emph{Complex analysis}, Graduate Texts in Mathematics, vol. 103
  (Springer-Verlag, New York, 1999), fourth edn.

\bibitem[Lei94]{MR1292806}
Leitenberger, F.
\newblock \emph{Unitary representations and coadjoint orbits for a group of
  germs of real analytic diffeomorphisms}.
\newblock Math. Nachr. \textbf{169} (1994):185--205.

\bibitem[Mil83]{milnor1983}
Milnor, J.
\newblock \emph{{Remarks on infinite-dimensional Lie groups}}.
\newblock In B.~DeWitt and R.~Stora (Eds.), \emph{Relativity, Groups and
  Topology II}, pp. 1007--1057 (North Holland, New York, 1983)

\bibitem[MMMKV15]{MMMV14}
McLachlan, R.~I., Modin, K., Munthe-Kaas, H. and Verdier, O.
\newblock \emph{B-series methods are exactly the affine equivariant methods}.
\newblock Numer. Math. (2015):1--24.


\bibitem[Muj86]{MR842435}
Mujica, J.
\newblock \emph{Complex analysis in {B}anach spaces}, North-Holland Mathematics
  Studies, vol. 120 (North-Holland Publishing Co., Amsterdam, 1986).
\newblock Notas de Matem{\'a}tica [Mathematical Notes], 107

\bibitem[Nee01]{MR1853240}
Neeb, K.-H.
\newblock \emph{Infinite-dimensional groups and their representations}.
\newblock In \emph{Infinite dimensional {K}\"ahler manifolds ({O}berwolfach,
  1995)}, DMV Sem., vol.~31, pp. 131--178 (Birkh\"auser, Basel, 2001)

\bibitem[Nee06]{neeb2006}
Neeb, K.
\newblock \emph{{Towards a Lie theory of locally convex groups}}.
\newblock Japanese Journal of Mathematics \textbf{1} (2006)(2):291--468

\bibitem[SW15]{SchmedingWockel15}
Schmeding, A. and Wockel, C.
\newblock \emph{{(Re)constructing Lie groupoids from their bisections and
  applications to prequantisation}} 2015.
\newblock \eprint{arXiv:1506.05415}

\end{thebibliography}
\end{document}